\documentclass[12pt,reqno]{amsart}
\usepackage{amsmath , amssymb }
\usepackage{graphics}
\usepackage{verbatim}

\usepackage[mathscr]{eucal}

\newdimen\unit\newdimen\psep\newcount\nd\newcount\ndx\newbox\dotb\newbox\ptbox
\newdimen\dx\newdimen\dy\newdimen\dxx\newdimen\dyy\newdimen\hgt
\newdimen\xoff\newdimen\yoff
\newcommand\clap[1]{\hbox to 0pt{\hss{#1}\hss}}
\newcommand\vdisk[1]{{\font\dotf=cmr10 scaled #1\dotf.}}
\newcommand\varline[2]{\setbox\dotb\hbox{\vdisk{#1}}\xoff=-.5\wd\dotb
\wd\dotb=0pt\yoff=-.5\ht\dotb\psep=#2\ht\dotb}
\newcommand\varpt[1]{\setbox\ptbox\clap{\vdisk{#1}}\setbox\ptbox
\hbox{\raise-.5\ht\ptbox\box\ptbox}}
\newcommand\cpt{\copy\ptbox}
\newcommand\point[3]{\rlap{\kern#1\unit\raise#2\unit\hbox{#3}}}
\newcommand\setnd[4]{\dx=#3\unit\advance\dx-#1\unit\divide\dx by\psep
\dy=#4\unit\advance\dy-#2\unit\divide\dy by\psep \multiply\dx
by\dx\multiply\dy by\dy\advance\dx\dy\nd=1\advance\dx-1sp
\loop\ifnum\dx>0\advance\dx-\nd sp\advance\nd1\advance\dx-\nd
sp\repeat}
\newcommand\dl[4]{{\setnd{#1}{#2}{#3}{#4}\dline{#1}{#2}{#3}{#4}\nd}}
\newcommand\dline[5]{{\nd=#5\hgt=#2\unit\dx=#3\unit\advance\dx-#1\unit
\divide\dx by\nd\dy=#4\unit\advance\dy-#2\unit\divide\dy by\nd
\advance\hgt\yoff\rlap{\kern#1\unit\kern\xoff\loop\ifnum\nd>1\advance\nd-1
\advance\hgt\dy\kern\dx\raise\hgt\copy\dotb\repeat}}}
\newcommand\ellipse[4]{\qellip{#1}{#2}{#3}{#4}\qellip{#1}{#2}{#3}{-#4}%
\qellip{#1}{#2}{-#3}{#4}\qellip{#1}{#2}{-#3}{-#4}}
\newcommand\qellip[4]{{\setnd{0}{0}{#3}{#4}\dx=\unit\dy=0pt\raise\yoff\rlap{%
\kern#1\unit\kern\xoff\raise#2\unit\hbox{\loop\ifnum\dx>0\rlap{\kern#3\dx
\raise#4\dy\copy\dotb}\hgt=\dx\divide\hgt
by\nd\advance\dy\hgt\hgt=\dy \divide\hgt
by\nd\advance\dx-\hgt\repeat\rlap{\raise#4\dy\copy\dotb}}}}}
\newcommand\bez[6]{{\setnd{#1}{#2}{#3}{#4}\ndx=\nd\setnd{#3}{#4}{#5}{#6}
\ifnum\ndx>\nd\nd=\ndx\fi\dx=#3\unit\advance\dx-#1\unit\dy=#4\unit
\advance\dy-#2\unit\dxx=#5\unit\advance\dxx-#1\unit\dyy=#6\unit\advance
\dyy-#2\unit\advance\dxx-2\dx\advance\dyy-2\dy\divide\dxx
by\nd\divide\dyy
by\nd\advance\dx.25\dxx\advance\dy.25\dyy\divide\dx
by\nd\divide\dy by\nd \multiply\nd
by2\dx=100\dx\dy=100\dy\dxx=100\dxx\dyy=100\dyy\divide\dxx by\nd
\divide\dyy
by\nd\hgt=#2\unit\raise\yoff\rlap{\kern#1\unit\kern\xoff
\raise\hgt\copy\dotb\loop\ifnum\nd>0\advance\nd-1\advance\hgt0.01\dy
\kern0.01\dx\raise\hgt\copy\dotb\advance\dx\dxx\advance\dy\dyy\repeat}}}
\newcommand\ptu[3]{\point{#1}{#2}{\cpt\raise1ex\clap{$\scriptstyle{#3}$}}}
\newcommand\ptd[3]{\point{#1}{#2}{\cpt\raise-1.8ex\clap{$\scriptstyle{#3}$}}}
\newcommand\ptr[3]{\point{#1}{#2}{\cpt\raise-.4ex\rlap{$\ \scriptstyle{#3}$}}}
\newcommand\ptl[3]{\point{#1}{#2}{\cpt\raise-.4ex\llap{$\scriptstyle{#3}\ $}}}
\newcommand\ptlu[3]{\point{#1}{#2}{\raise.8ex\clap{$\scriptstyle{#3}$}}}
\newcommand\ptld[3]{\point{#1}{#2}{\raise-1.6ex\clap{$\scriptstyle{#3}$}}}
\newcommand\ptlr[3]{\point{#1}{#2}{\raise-.4ex\rlap{$\,\scriptstyle{#3}$}}}
\newcommand\ptll[3]{\point{#1}{#2}{\raise-.4ex\llap{$\scriptstyle{#3}\,$}}}
\newcommand\pt[2]{\point{#1}{#2}{\cpt}}

\newcommand\medline{\varline{800}{.5}}
\newcommand\thnline{\varline{400}{.6}}

\varpt{2500}\thnline\unit=2em

\newtheorem{thm}{Theorem}

\newtheorem{prob}{Problem}
\newtheorem{lemma}[thm]{Lemma}
\newtheorem{prop}[thm]{Proposition}

\newtheorem{obs}[thm]{Observation}
\theoremstyle{definition}
\theoremstyle{definition}\newtheorem{defn}{Definition}
\newtheorem*{cliquealg}{The Clique Process}
\newtheorem*{WSA}{The Witness-Set Algorithm}
\newtheorem*{REA}{The Red Edge Algorithm}
\newcommand{\ds}{\displaystyle}

\def\A{\mathcal{A}}

\def\G{\mathcal{G}}

\def\J{\mathcal{J}}
\def\K{\mathcal{K}}

\def\P{\mathcal{P}}

\def\T{\mathcal{T}}

\def\Ex{\mathbb{E}}

\def\N{\mathbb{N}}
\def\Pr{\mathbb{P}}

\def\ZZ{\mathbb{Z}}

\newcommand{\1}{\mathbf{1}}

\def\le{\leqslant}
\def\ge{\geqslant}

\def\eps{\varepsilon}

\def\Var{\textup{Var}}

\def\root{\textup{root}}

\def\<{\langle}
\def\>{\rangle}

\begin{document}
\title{Graph bootstrap percolation}

\author{J\'ozsef Balogh}
\address{Department of Mathematics\\ University of Illinois\\ 1409 W. Green Street\\ Urbana, IL 61801\\ and\\ Department of Mathematics\\ University of California\\ San Diego, La Jolla, CA 92093}\email{jobal@math.uiuc.edu}

\author{B\'ela Bollob\'as}
\address{Trinity College\\ Cambridge CB2 1TQ\\ England\\ and \\ Department of Mathematical Sciences\\ The University of Memphis\\ Memphis, TN 38152, USA} \email{B.Bollobas@dpmms.cam.ac.uk}

\author{Robert Morris}
\address{IMPA, Estrada Dona Castorina 110, Jardim Bot\^anico, Rio de Janeiro, RJ, Brasil} \email{rob@impa.br}
\thanks{Research supported in part by: (JB) NSF CAREER Grant DMS-0745185, UIUC Campus Research Board Grants 09072 and 11067, OTKA Grant K76099, and the TAMOP-4.2.1/B-09/1/KONV-2010-0005 project; (BB)  NSF grants DMS-0906634, CNS-0721983 and CCF-0728928, ARO grant W911NF-06-1-0076, and TAMOP-4.2.2/08/1/2008-0008 program of the Hungarian Development Agency; (RM) CNPq bolsa de Produtividade em Pesquisa}

\begin{abstract}
Graph bootstrap percolation is a deterministic cellular automaton which was introduced by Bollob\'as in 1968, and is defined as follows. Given a graph $H$, and a set $G \subset E(K_n)$ of initially `infected' edges, we infect, at each time step, a new edge $e$ if there is a copy of $H$ in $K_n$ such that $e$ is the only not-yet infected edge of $H$. We say that $G$ percolates in the $H$-bootstrap process if eventually every edge of $K_n$ is infected. The extremal questions for this model, when $H$ is the complete graph $K_r$, were solved (independently) by Alon, Kalai and Frankl almost thirty years ago. In this paper we study the random questions, and determine the critical probability $p_c(n,K_r)$ for the $K_r$-process up to a poly-logarithmic factor. In the case $r = 4$ we prove a stronger result, and determine the threshold for $p_c(n,K_4)$.
\end{abstract}

\maketitle

\section{Introduction}\label{intro}

Cellular automata, which were introduced by von Neumann (see~\cite{vN}) after a suggestion of Ulam~\cite{Ulam}, are dynamical systems (defined on a graph $G$) whose update rule is homogeneous and local. We shall study a particular cellular automaton, called $H$-bootstrap percolation, which was introduced over 40 years ago by Bollob\'as~\cite{Bela68}. This model is a substantial generalization of $r$-neighbour bootstrap percolation (see below), an extensively studied model related to statistical physics. We shall determine the critical probability for $K_r$-percolation up to a poly-logarithmic factor for every $r \ge 4$ and moreover, using a completely different method, we shall determine the threshold for percolation in the case $r = 4$.

Given a graph $H$, we define \emph{$H$-bootstrap percolation} (or \emph{$H$-edge-bootstrap percolation}) as follows. Given a set $G \subset E(K_n)$ of initially `infected' edges on vertex set $[n]$ (that is, given a graph), we set $G_0 = G$ and define, for each $t \ge 0$,
$$G_{t+1} \, := \,  G_t \cup \Big\{ e \in E(K_n) \,:\, \exists \, H \text{ with } e \in H \subset G_t \cup \{e\}  \Big\}.$$
In words, this says that an edge $e$ becomes infected at time $t + 1$ if there exists a copy of $H$ in $K_n$ for which $e$ is the only uninfected edge at time $t$. Let $\< G \>_H = \bigcup_t G_t$ denote the closure of $G$ under the $H$-bootstrap process, and say that $G$ \emph{percolates} (or $H$-percolates) in $K_n$ if $\< G \>_H = E(K_n)$.

The $H$-bootstrap process was introduced over 40 years ago by Bollob\'as~\cite{Bela68} (see also~\cite{Bela11}), under the name `weak saturation'. He conjectured that if a graph $G$ percolates in the $K_r$-process, then $G$ has at least $\binom{n}{2} - \binom{n-r+2}{2}$ edges, and, building on work in~\cite{Bela65}, proved his conjecture when $r \le 7$. For general $r$, the conjecture was proved using linear algebraic methods by Alon~\cite{Alon}, Frankl~\cite{Fra} and Kalai~\cite{Kalai}. See~\cite{BBMR} for more recent extremal results, on a closely related process, using such methods.

In this paper, we shall study the $H$-bootstrap process in the random setting, i.e., when the initial graph $G$ is chosen to be $G_{n,p}$. Apart from its intrinsic interest, this question is motivated by the following, closely related cellular automaton, which was introduced in 1979 by Chalupa, Leath and Reich~\cite{CLR} in the context of disordered magnetic systems, and for which our process is named. Given an underlying graph $G$, an integer $r$ and a set of infected vertices $A \subset V(G)$, set $A_0 = A$ and let
$$A_{t+1} \, := \,  A_t \cup \big\{v \in V(G) : |N(v) \cap A_t| \ge r \big\}$$
for each $t \ge 0$; that is, we infect a vertex if it has at least $r$ already-infected neighbours. Say that the set $A$ \emph{percolates} if the entire vertex set is eventually infected. This process is known as $r$-neighbour bootstrap percolation, and has been extensively studied by mathematicians (see, for example,~\cite{AL,BBDM,CC,Hol,HLR,Sch}), physicists (see~\cite{ALev}, and the references therein) and sociologists~\cite{Gran,Watts}, amongst others. It has moreover found applications in the Glauber Dynamics of the Ising model (see~\cite{FSS,M}).

The $r$-neighbour bootstrap model is usually studied in the random setting, where the main question is to determine the \emph{critical threshold} at which percolation occurs. To be precise, if $V(G) = [n]$ and the elements of $A \subset V(G)$ are chosen independently at random, each with probability $p$, then one aims to determine the value $p_c$ of $p = p(n)$ at which percolation becomes likely. Sharp bounds on $p_c$ have recently been determined in several cases of particular interest, such as $[n]^d$ (see~\cite{BBDM,BBMMaj,BBMdr3,BBMhigh,GHM,Hol}), on a large family of `two-dimensional' graphs~\cite{DCH}, on trees~\cite{BPP,FS}, and on various types of random graph~\cite{BP, JLTV}. In each case, it was shown that the critical probability has a sharp threshold.

Motivated by these results, let us define the critical threshold for $H$-bootstrap percolation on $K_n$ as follows:
$$p_c(n,H) \, := \, \inf\Big\{ p : \Pr\big( \< G_{n,p} \>_H = K_n \big) \ge 1/2 \Big\},$$
where $G_{n,p}$ is the Erd\H{o}s-R\'enyi random graph, obtained by choosing each edge independently with probability $p$. (For background on the theory of Random Graphs, see~\cite{RG}.) We remark that, by a general result of Bollob\'as and Thomason~\cite{BT}, the event $\< G_{n,p} \>_H = K_n$ has a threshold, i.e., if $p \ll p_c(n,H)$ then the probability of percolation is $o(1)$, and if $p \gg p_c(n,H)$ then it is $1 - o(1)$. Moreover, a general result of Friedgut~\cite[Theorem~1.4]{Fri}, combined with Theorem~\ref{k=4}, below, shows that this event has a sharp threshold\footnote{An event $A$ has a \emph{sharp} threshold if the `window' (in $p$) in which $A$ has probability between $\eps$ and $1 - \eps$ has size $o(p_c)$; otherwise it has a \emph{coarse} threshold.} when $H = K_4$, and we expect this to hold for all $K_r$. However, it is not hard to see that if $H = K_r + e$ (i.e., $H$ is $K_r$ plus a pendant edge) then the events $\< G_{n,p} \>_H = K_n$ and $K_r - e \subset G_{n,p}$ differ by a set of measure $o(1)$ at $p = p_c$, so in this case the event has a coarse threshold. 

Our aim is to determine $p_c(n,H)$ for every graph $H$. Here we shall study the case $H = K_r$, the complete graph; our main theorems partially solve Problem~1 of~\cite{Bela11}. In order to aid the reader's intuition, let us first consider the case $H = K_3$, which follows easily from classical results. Indeed, it is easy to see that $G$ percolates in the $K_3$-process if and only if $G$ is connected. It is well-known (see~\cite{RG}) that, with high probability, $G_{n,p}$ is connected if and only if it has no isolated vertex; thus, a straightforward calculation gives the following theorem of Erd\H{o}s and R\'enyi~\cite{ER}, which was one of the first results on random graphs:
$$p_c(n,K_3) \, = \,  \frac{\log n}{n} \,+\, \Theta\bigg( \frac{1}{n} \bigg).$$
In fact Erd\H{o}s and R\'enyi proved even more: that if $p = (\log n + c)/n$, then the probability that $G_{n,p}$ percolates in the $K_3$-process converges to $e^{-e^{-c}}$ as $n \to \infty$. We remark that the same result holds for the $C_k$-process for any $k \ge 3$, see Section~\ref{othersec}.

For $r \ge 4$, the problem is more challenging, since there seems to be no simple description of the closed sets under the $K_r$-process. Set
$$\lambda(r) \,:=\, \frac{{r \choose 2} - 2 }{r - 2}.$$
The following theorem is our main result.

\begin{thm}\label{thm:Kr}
For every $r \ge 4$, there exists a constant $c = c(r) > 0$ such that
$$\frac{n^{-1/\lambda(r)}}{c \log n} \, \le \,  p_c(n,K_r) \, \le \,  n^{-1/\lambda(r)} \log n$$
for every sufficiently large $n \in \N$.
\end{thm}

In fact we shall prove slightly stronger bounds (see Propositions~\ref{prop:bal} and~\ref{prop:lower}); however, we do not expect either of our bounds to be sharp. The proof of the lower bound in Theorem~\ref{thm:Kr} is based on an extremal result on graphs which cause a given edge to be infected (see Lemma~\ref{extremal}). Although it is not long, the proof of this lemma is delicate, and does not seem to extend easily to other graphs. The upper bound, on the other hand, holds for a much wider family of graphs $H$ (see Section~\ref{uppersec}), which we call `balanced'.

In the case $r = 4$ we shall prove the following stronger result, which determines the sharp threshold of $p_c(n,K_4)$ up to a constant\footnote{In the published version of this article, we stated a slightly stronger upper bound than that claimed here. This was due to a small error in the proof of Proposition~\ref{k4:upper}, below. We would like to thank Brett Kolesnik for pointing out this error to us.}  factor.

\begin{thm}\label{k=4}
If $n$ is sufficiently large, then
$$ \frac{1}{4} \sqrt{\frac{1}{ n \log n} } \, \le \,  p_c(n,K_4) \, \le \,  24 \sqrt{ \frac{1}{n \log n} }.$$
\end{thm}

The proof of Theorem~\ref{k=4} is completely different from that of Theorem~\ref{thm:Kr}, and is based on ideas from two-neighbour bootstrap percolation on $[n]^d$. 

The rest of the paper is organized as follows. In Sections~\ref{uppersec} and~\ref{lowersec} we shall prove the upper and lower bounds in Theorem~\ref{thm:Kr}, respectively. In Section~\ref{k4sec} we shall prove Theorem~\ref{k=4}, and in Section~\ref{othersec} we shall discuss other graphs $H$, and state some open problems.

\section{An upper bound for balanced graphs}\label{uppersec}

In this section we shall prove the upper bound in Theorem~\ref{thm:Kr}; in fact we prove a stronger bound for a more general family of graphs, $H$. Throughout, we shall assume that $v(H) \ge 4$, since otherwise the problem is trivial. We make the following definition.

\begin{defn}\label{def:bal}
We call a graph $H$ \emph{balanced} if $e(H) \ge 2v(H) - 2$, and
$$\frac{e(F) - 1}{v(F) - 2} \, \le \,  \lambda(H) \, := \,  \frac{e(H)-2}{v(H)-2}$$
for every proper subgraph $F \subset H$ with $v(F) \ge 3$.
\end{defn}

It is straightforward to check that the complete graph $K_r$ is balanced for every $r \ge 4$. Thus, the upper bound in Theorem~\ref{thm:Kr} follows immediately from the following proposition.

\begin{prop}\label{prop:bal}
If $H$ is a balanced graph, then
$$p_c(n,H) \, \le \,  C \left( \frac{\log n}{\log \log n} \right)^{2 / \lambda(H)} n^{-1 / \lambda(H)},$$
for some constant $C = C(H) > 0$.
\end{prop}

Note that $\lambda(K_r) = \lambda(r)$ is increasing, and satisfies
\begin{equation}
\frac{r}{2} \,\le\, \lambda(r) \,\le\, \frac{r+1}{2}
\end{equation}
if $r \ge 4$ (we shall use these bounds several times during the proof), so Proposition~\ref{prop:bal} actually implies the following slightly stronger upper bound than that stated in Theorem~\ref{thm:Kr}:
$$p_c(n,K_r) \, \le \,   n^{-1 / \lambda(H)} \big( \log n \big)^{4 / r}.$$

We begin by sketching the proof of Proposition~\ref{prop:bal}. We shall describe one way in which an edge can be infected after $d$ steps, and then show that (with high probability) most edges will be infected in this way (if they are not already infected sooner). Indeed, for each $d \in \N$ we shall define a graph $H_d$ with $(v(H) - 2)d + 2$ vertices and $(e(H) - 2)d + 1$ edges, and an edge $e \in {V(H_d) \choose 2} \setminus E(H_d)$ (which we call the \emph{root} of $H_d$) such that $e \in \< H_d \>_H$, and $H_d$ is minimal subject to this condition. In other words, $H_d$ causes $e$ to be infected in the $H$-bootstrap process, and no subgraph of $H_d$ has this property.

To define the graphs $H_d$, first choose a sequence of edges $(e_1,e_2,\ldots)$ of $H$ such that  for every $j \in \N$, $e_j$ and $e_{j+1}$ do not share an endpoint. Let $(H^{(1)}, H^{(2)}, \ldots, H^{(d)})$ be a sequence of copies of $H$ and, for each $1 \le j \le d-1$, identify the endpoints of the edge $e_{j+1}$ in $H^{(j)}$ and $H^{(j+1)}$ (see Figure~1). Finally, remove the edge $e_1$ from $H^{(1)}$ and, for each $1 \le j \le d-1$, remove the edge $e_{j+1}$ from $H^{(j)} \cap H^{(j+1)}$.

\vskip0.1in
\[ \unit = 0.4cm \hskip -16\unit
\medline \ellipse{0}{0}{3}{2} \ellipse{4}{0}{3}{2} \ellipse{8}{0}{3}{2} \ellipse{12}{0}{3}{2} \ellipse{16}{0}{3}{2}
\varpt{3000}
\point{0}{0}{$ \pt{2}{1} \pt{2}{-1}  \dl{2}{0.9}{2}{0.7}  \dl{2}{0.5}{2}{0.3}   \dl{2}{0.1}{2}{-0.1}   \dl{2}{-0.3}{2}{-0.5} \dl{2}{-0.7}{2}{-0.9}  \point{2.2}{-0.1}{\tiny $e_2$} \point{3.25}{0.75}{\small $H^{(2)}$}  $}
\point{4}{0}{$ \pt{2}{1} \pt{2}{-1}  \dl{2}{0.9}{2}{0.7}  \dl{2}{0.5}{2}{0.3}   \dl{2}{0.1}{2}{-0.1}   \dl{2}{-0.3}{2}{-0.5} \dl{2}{-0.7}{2}{-0.9} \point{2.2}{-0.1}{\tiny $e_3$} \point{3.25}{0.75}{\small $H^{(3)}$} $}
\point{8}{0}{$ \pt{2}{1} \pt{2}{-1}  \dl{2}{0.9}{2}{0.7}  \dl{2}{0.5}{2}{0.3}   \dl{2}{0.1}{2}{-0.1}   \dl{2}{-0.3}{2}{-0.5} \dl{2}{-0.7}{2}{-0.9}  \point{2.2}{-0.1}{\tiny $e_4$} \point{3.25}{0.75}{\small $H^{(4)}$}  $}
\point{12}{0}{$ \pt{2}{1} \pt{2}{-1}  \dl{2}{0.9}{2}{0.7}  \dl{2}{0.5}{2}{0.3}   \dl{2}{0.1}{2}{-0.1}   \dl{2}{-0.3}{2}{-0.5} \dl{2}{-0.7}{2}{-0.9} \point{2.2}{-0.1}{\tiny $e_5$} \point{3.45}{0.75}{\small $H^{(5)}$} $}
\point{-4}{0}{$ \pt{2}{1} \pt{2}{-1}  \dl{2}{0.9}{2}{0.7}  \dl{2}{0.5}{2}{0.3}   \dl{2}{0.1}{2}{-0.1}   \dl{2}{-0.3}{2}{-0.5} \dl{2}{-0.7}{2}{-0.9}  \point{2.2}{-0.1}{\tiny $e_1$} \point{3.1}{0.75}{\small $H^{(1)}$}  $}
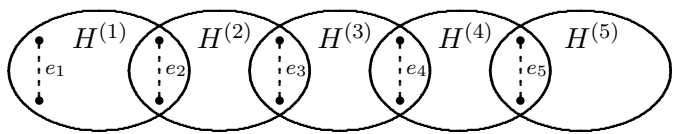
\point{3}{-4}{\small Figure 1: The graph $H_5$}
\]
\vskip0.2in

For each $1 \le j \le d$, let $V_j = V(H_j)$, so $|V_i \cap V_j| = 2$ if $|i - j| = 1$, and $|V_i \cap V_j| = 0$ otherwise. Note that $E(H_d) = E( H_d[V_1]) \cup \ldots \cup E( H_d[V_d] )$, that $E\big( H_d[V_j] \big) = E(H) \setminus \big\{ e_j,e_{j+1} \big\}$ for every $1 \le j \le d-1$, and that $E\big( H_d[V_d] \big) = E(H) \setminus \{ e_d \}$. Finally, set $\root(H_d) = e_1$. We remark that although the definition of $H_d$ depends on the choice of $(e_1,e_2,\ldots)$, the proof below will work for any such sequence.

We shall use the following simple properties of $H_d$.

\begin{obs}
For every $d \in \N$,
$$(a) \; v(H_d) = (v(H) - 2)d + 2 \qquad (b) \; e(H_d) = (e(H) - 2)d + 1 \qquad (c) \; \root(H_d) \in \< H_d \>_H.$$
\end{obs}

\begin{proof}
Properties~$(a)$ and~$(b)$ follow immediately from the definition, since the edge sets of $H_d[V_j]$ are all disjoint. To prove $(c)$, simply note that the edge $e_{d-j+1}$ is infected after $j$ steps of the $H$-bootstrap process on $H_d$, since it completes a copy of $H$ on vertex set $V_{d-j+1}$, and hence $e_1 \in \< H_d \>_H$, as claimed.
\end{proof}

Let $X_d(e)$ be the random variable which counts the number of copies of $H_d$ in $G_{n,p}$, rooted at a given edge $e \in E(K_n)$. It is straightforward (using properties~$(a)$ and~$(b)$) to show that the expected value of $X_d$ is large if $p^{\lambda(H)} n \ge  (\log n)^2$ (see Lemma~\ref{expX}); the main challenge will be to bound the variance of $X_d$. The key step is therefore Lemma~\ref{var:ext}, below, which controls the number of edges in the intersection of two copies of $H_d$ with the same root: this will enable us to prove Lemma~\ref{varX}.
Having bounded the variance of $X_d$, the proposition follows easily by Chebyshev's inequality.

We begin by bounding the expected value of the counting function $X_d(e)$.

\begin{lemma}\label{expX}
Let $H$ be a balanced graph, and $e \in E(K_n)$. If $p = p(n)$ and $d = d(n)$ are chosen so that $p^{\lambda(H)} n \ge \omega v(H)d$ and $\omega^{(v(H) - 2)d} \ge n$ for some function $\omega = \omega(n)$, and $pn \to \infty$, then
$$\Ex\big( X_d(e) \big) \, \to \,  \infty$$
as $n \to \infty$.
\end{lemma}

\begin{proof}
Recall that $H_d$ has $(v(H) - 2)d + 2$ vertices and $(e(H) - 2)d + 1$ edges. Thus
\begin{equation*}\label{eq:ex}
\Ex\big( X_d(e) \big) \, \ge \,  {n \choose {v(H_d)-2}} p^{e(H_d)} \, \ge \,  \left( \frac{n}{v(H)d} \right)^{(v(H) - 2)d} p^{(e(H) - 2)d + 1}.
\end{equation*}
Since $e(H) - 2 = \lambda(H) (v(H)-2)$, and using our bounds on $\omega$, it follows that
\begin{equation*}
\Ex\big( X_d(e) \big) \, \ge \,  p \left( \frac{p^{\lambda(H)} n }{v(H)d} \right)^{(v(H) - 2)d} \, \ge \,  p \cdot \omega^{(v(H) - 2)d} \, \ge \,  pn \, \to \,  \infty,
\end{equation*}
as required.
\end{proof}

We shall next bound the variance of $X_d(e)$; the following lemma is the key step.

\begin{lemma}\label{var:ext}
Let $H$ be a balanced graph, and let $d \in \N$. If $F \subsetneq H_d$ contains the endpoints of the root of $H_d$, then
$$e(F) \, \le \,  \big( v(F) - 2 \big) \lambda(H).$$
\end{lemma}

\begin{proof}
We shall use induction on $d$. The case $d = 1$ is trivial, since $F$ contains the endpoints of the root of $H_d$, so either $v(F) = 2$ and $e(F) = 0$, or $v(F) \ge 3$ and $F + e \subsetneq H$, in which case the bound follows by Definition~\ref{def:bal}. Let $d \ge 2$, and assume that the result holds for every $d' < d$. For each $j \in [d]$, let $F_j = H^{(j)}[V(F) \cap V_j]$ be the subgraph of $H^{(j)}$ induced by the vertices of $F$.

Suppose first that $v(F_j) \le 1$ for some $j \in [d]$, and let $F'$ and $F''$ be the subgraphs of $H_d$ induced by $V(F) \cap \big( V_1 \cup \ldots \cup V_{j-1} \big)$ and $V(F) \cap \big( V_{j+1} \cup \ldots \cup V_{d} \big)$, respectively. Applying the induction hypothesis to $F'$, we see that
$$e(F) \, = \,  e(F') + e(F'') \, \le \,  \big( v(F') - 2 \big) \lambda(H) + e(F''),$$
so it will suffice to prove that $e(F'') \le v(F'') \lambda(H)$. Now, applying the induction hypothesis to $F^*$, the subgraph of $H_d$ induced by $V(F'') \cup (V_j \cap V_{j+1})$, we get either
$$e( F^* ) \, \le \,  \big( v(F^*) - 2 \big) \lambda(H) \, \le \,  v(F'') \lambda(H)$$
as required, or $V(F^*) = V_{j+1} \cup \ldots \cup V_d$. But in the latter case $e(F^*) - e(F'') \ge 1$, since $v(F_j) \le 1$ implies that $v(F^*) > v(F'')$, and $\delta(H) \ge 2$ (since $H$ is balanced), so the new vertex adds at least one new edge. It follows that
$$e(F'') \, \le \,  e(F^*) - 1 \, = \,  \big( v(F^*) - 2 \big) \lambda(H) \, \le \,  v(F'') \lambda(H),$$
as required. Hence we may assume that $v(F_j) \ge 2$ for every $j \in [d]$.

Next, suppose that $V(F) = V(H_d)$. In this case the lemma is easy, since $F \neq H_d$ (by assumption), and so
$$e(F) \, \le \, e(H_d) - 1 \, = \, \big( e(H) - 2 \big) d \, = \, \big( v(H) - 2 \big) \lambda(H) \cdot d \, = \, \big( v(F) - 2 \big) \lambda(H)$$
as required, since $v(F) - 2 = v(H_d) - 2 = (v(H) - 2)d$.

Thus we may assume that $v(F) < v(H_d)$, and that $v(F_j) \ge 2$ for every $j \in [d]$. Now, for each $j \in [d-1]$, let $E_j$ denote the event that $V_j \cap V_{j+1} \subset V(F)$, and let $\1[\cdot]$ denote the indicator function. Then, recalling that $F_j = H^{(j)}[V_j \cap V(F)]$, we have
$$e(F) \, \le \,  \bigg( \sum_{j=1}^d e(F_j) \bigg) - 1 - 2\sum_{j=1}^{d-1} \1[E_j],$$
by the definition of $H_d$, and since $F$ contains the endpoints of the root of $H_d$. We next claim that, since $H$ is balanced and $v(F) < v(H_d)$, it follows that
\begin{eqnarray*}
e(F) & \le & \sum_{j=1}^d \Big( \big( v(F_j) - 2 \big) \lambda(H) + 2 \Big) - 2 - 2 \sum_{j=1}^{d-1} \1[E_j]\\
& = & \bigg( \sum_{j=1}^d v(F_j) - 2d \bigg) \lambda(H) \,+\, \big( 2d - 2 \big) \,-\, 2 \sum_{j=1}^{d-1} \1[E_j].
\end{eqnarray*}
To see this, observe that $e(F_j) \le \big( v(F_j) - 2 \big) \lambda(H) + 1$ holds for every $F_j \subsetneq H$ with $v(F_j) \ge 2$, by Definition~\ref{def:bal}, and that $F_j \neq H$ for some $j \in [d]$, since $v(F) < v(H_d)$.

Finally, observe that
$$v(F) \, \ge \,  \sum_{j=1}^d v(F_j) - (d-1) - \sum_{j=1}^{d-1} \1[E_j],$$
and so
\begin{eqnarray*}
e(F) \, \le \,  \left(v(F) - d - 1 + \sum_{j=1}^{d-1} \1[E_j] \right) \lambda(H) \,+\, \big( 2d - 2 \big) \,-\, 2 \sum_{j=1}^{d-1} \1[E_j].
\end{eqnarray*}
But
$$\left( d - 1 - \sum_{j=1}^{d-1} \1[E_j]  \right) \lambda(H) \, \ge \,  2d - 2 - 2 \sum_{j=1}^{d-1} \1[E_j],$$
since $\lambda(H) \ge 2$, by Definition~\ref{def:bal}. Hence
\begin{eqnarray*}
e(F) \, \le \,  \big( v(F) - 2 \big) \lambda(H)
\end{eqnarray*}
for every $F \subsetneq H$, as required.
\end{proof}

It is now straightforward to deduce the required bound on the variance of the counting function $X_d(e)$.

\begin{lemma}\label{varX}
Let $H$ be a balanced graph, and $e \in E(K_n)$. If $p = p(n)$ and $d = d(n)$ are chosen so that $v(H_d)^{-2} p^{\lambda(H)} n \to \infty$ as $n \to \infty$, then
$$\frac{\Var\big( X_d(e) \big)}{\Ex\big( X_d(e) \big)^2} \, \to \,  0$$
as $n \to \infty$.
\end{lemma}

\begin{proof}
Let $\ell(H_d)$ denote the number of copies of $H_d$, rooted at $e$, which have the same vertex set. Then
$$\Ex\big( X_d(e) \big) \, = \,  {n \choose {v(H_d)-2}} \ell(H_d) \cdot p^{e(H_d)}.$$
Moreover, we claim that
\begin{equation}\label{eq:var}
\Var\big( X_d(e) \big) \, \le \,  \sum_{m=1}^{v(H_d)-2} \ell(H_d)^2 {n \choose m} {n \choose {v(H_d)-m-2}}^2 p^{2e(H_d) - \lambda(H) m}.
\end{equation}
To see this, we simply count (ordered) pairs $(A,B)$, where $A$ and $B$ are copies of $H_d$ in $G_{n,p}$ with root $e$. Let $F = H_d[A \cap B]$ and $m = |V(A) \cap V(B)| - 2$, so $F$ is the intersection of the edge sets of $A$ and $B$, and $m$ is the number of vertices in their intersection, not counting the endpoints of $e$. Note that we expect at most $\Ex\big( X_d(e) \big)^2$ such pairs $(A,B)$ with $m = 0$.

By Lemma~\ref{var:ext}, if $A \neq B$ then $e(F) \le \lambda(H) m$, and so $e(A \cup B) \ge 2e(H_d) - \lambda(H) m$. Moreover, given $m$, there are at most
$$\ell(H_d)^2 {n \choose m} {n \choose {v(H_d)-m-2}}^2$$
choices for $A$ and $B$. This proves~\eqref{eq:var}.

Combining the bounds above, and setting $k = v(H_d)$, it follows that $\Var( X_d(e))\big/\Ex( X_d(e))^2$ is at most
$$\sum_{m=1}^{k-2} \frac{{n \choose m} {n \choose {k-m-2}}^2}{{n \choose {k-2}}^2} \cdot p^{ - \lambda(H) m} \, \le \, \sum_{m=1}^{k-2} \frac{n^m}{m!} \cdot \frac{(k-2)!^2}{(k-m-2)!^2} \cdot \left( \frac{1}{n-k} \right)^{2m} \cdot p^{ - \lambda(H) m}.$$
Since $k = o(n)$ (by assumption), it follows that
$$\frac{\Var\big( X_d(e) \big)}{\Ex\big( X_d(e) \big)^2} \, \le \, \sum_{m=1}^{k-2} 2^m \cdot m! {k-2 \choose m}^2 \left( \frac{2}{p^{\lambda(H)} n} \right)^{m} \, \le \, \sum_{m=1}^{k-2}  \left( \frac{4k^2}{p^{\lambda(H)} n} \right)^{m} \, \to \, 0$$
as $n \to \infty$, as required.
\end{proof}

We can now deduce Proposition~\ref{prop:bal} using Chebyshev's inequality and sprinkling.

\begin{proof}[Proof of Proposition~\ref{prop:bal}]
Let $H$ be a balanced graph, suppose that $p \gg \left( \frac{\log n}{\log\log n} \right)^{2 / \lambda(H)} n^{-1 / \lambda(H)}$, and let
$$d(n) \,=\, \left\lfloor \frac{\log n}{\log \log n} \right\rfloor.$$
We claim that $p(n)$ and $d(n)$ satisfy the conditions of Lemmas~\ref{expX} and~\ref{varX}. Indeed, setting $\omega(n) = d(n)$ we have $p^{\lambda(H)} n \gg \omega d$ and $\omega^{2d} \gg n$, so Lemma~\ref{expX} holds, and $p^{\lambda(H)} n \gg d^2$, so Lemma~\ref{varX} holds. Thus, by Chebyshev's inequality,
$$\Pr\big( X_d(e) = 0 \big) \, \le \, \frac{\Var\big( X_d(e) \big)}{\Ex\big( X_d(e) \big)^2} \, \to \, 0$$
as $n \to \infty$. Moreover, if $X_d(e) \neq 0$ then $e \in \< G_{n,p} \>_H$, since if $e$ is the root of some copy of $H_d$ then it is infected after at most $d$ steps of the $H$-process. Hence, by Markov's inequality, if $p^{\lambda(H)} n \gg \left( \frac{\log n}{\log\log n} \right)^2$ then, with high probability, all but $o(n^2)$ edges of $K_n$ are infected in the $H$-process on $G_{n,p}$.

To finish the proof, we shall show that by sprinkling $O(n \log n)$ extra edges, we shall infect all of the remaining edges, with high probability. We use the following easy claim.

\medskip
\noindent \textbf{Claim:} If $v(G) = n$ and $e(G) \ge {n \choose 2} - o(n^2)$, then there is a clique of size $n - o(n)$ in $\< G \>_H$.

\begin{proof}[Proof of Claim]
Let $0 < c < 1/2$ be arbitrary, and let
$$D \, := \, \big\{ x \in V(G) \,:\, d_G(x) > (1-c)n \big\}.$$
By our assumption, $|D| = n - o(n)$; we claim that $D$ is a clique in $\< G \>_H$.

Indeed, if $x,y \in D$ then by Tur\'an's Theorem there is a $(v(H)-2)$-clique in $N_G(x) \cap N_G(y)$, since $|N_G(x) \cap N_G(y)| \ge (1 - 2c) n$ and $o(n^2)$ edges are missing. But then $xy \in \< G \>_H$, and so $D$ is a clique of size $n - o(n)$ in $\< G \>_H$, as claimed.
\end{proof}

Continuing our proof of Proposition~\ref{prop:bal},
let us sprinkle edges with density $p$; that is, let us take a second copy of $G_{n,p}$ and consider the union of the two random graphs. We obtain a random graph $G_{n,p^*}$ of density $p^* = 1 - (1-p)^2 < 2p$. Let $K$ be the clique found in the claim, and observe that if every vertex outside $K$ has at least $v(H) - 1$ neighbours in $K$ (in the second copy of $G_{n,p}$) then $G_{n,p^*}$ will percolate. Since $pn \gg \log n$, this occurs with high probability, and hence
$$p_c(n,H) \, \le \,  C \left( \frac{\log n}{\log \log n} \right)^{2 / \lambda(H)} n^{-1 / \lambda(H)},$$
if $C = C(H)$ is sufficiently large, as required.
\end{proof}

\section{Lower bound for $K_r$-percolation}\label{lowersec}

In this section we shall prove the following proposition, which shows that, if $r \ge 4$ and $(p \log n)^{\lambda(r)} n = o(1)$, then with high probability $o(n^2)$ edges are infected in the $K_r$-bootstrap process with initial set $G_{n,p}$.

\begin{prop}\label{prop:lower}
Let $r \ge 4$, and let $e \in E(K_n)$. If $pn^{1/\lambda(r)} \log n \le 1/(2e)$, then
$$\Pr\Big( e \in \< G_{n,p} \>_{K_r} \Big) \,\to\, 0$$
as $n \to \infty$.
\end{prop}

The idea of the proof is as follows. If $e \in \< G \>_{K_r}$ for some graph $G$, then there must exist a `witness set' of edges of $G$ which caused $e$ to be infected. We shall describe an algorithm which finds such a set $F = F(e)$ of edges, and show that this set has two useful properties:\smallskip
\begin{itemize}
\item[$(a)$] $e(F) \ge \lambda(r) \big( v(F) - 2 \big) + 1$ (see Lemma~\ref{extremal}).\\[-1.5ex]
\item[$(b)$] If $e(F) \ge {r \choose 2} L$, then $L \le e(F(f)) \le {r \choose 2}L$ for some $f \in \< G \>_{K_r}$ (see Lemma~\ref{ALlem}).
\end{itemize}
\smallskip
Property $(a)$ will allow us to bound the expected number of such sets when $G = G_{n,p}$ and $e(F) = O(\log n)$; combining it with property $(b)$ will allow us to do so when $e(F)$ is larger than this.

\subsection{Extremal results}

Let $r \ge 4$ be fixed for the remainder of this section, and let $G$ be an arbitrary graph. We begin by describing the algorithm which finds $F(e)$.

\begin{WSA}
We assign a graph $F = F(e) \subset G$ to each edge $e \in \< G \>_{K_r}$ as follows:
\begin{enumerate}
\item[1.] If $e \in G$ then set $F(e) = \{e\}$.
\item[2.] Choose an order in which to infect the edges of $\< G \>_{K_r}$, and at each step identify which $r$-clique was completed (if more than one is completed then choose one).
\item[3.] Infect the edges one by one. If $e$ is infected by the $r$-clique $K$, then set
$$F(e) \, := \, \bigcup_{e \ne e' \in K} F(e').$$
\end{enumerate}
We call the graph $F(e)$ a \emph{witness set} for the event $e \in \< G \>_{K_r}$.
\end{WSA}

Since every $e \ne e' \in K$ is either in $G$, or was infected earlier in the process, the algorithm is well-defined. Note that the graphs $F(e)$ depend on the order in which we chose to infect the edges (that is, they depend on Step 2 of the algorithm); the results below hold for every possible such choice.

We shall say that a graph $F$ is an $r$-\emph{witness set} if there exists a graph $G$, an edge $e$, and a realization of the Witness-Set Algorithm (i.e., a choice as in Step 2) such that $F = F(e)$. The key lemma in the proof of Proposition~\ref{prop:lower} is the following extremal result.

\begin{lemma}\label{extremal}
Let $F$ be a graph and $r \ge 4$, and suppose that $F$ is an $r$-witness set. Then
$$e(F) \, \ge \,  \lambda(r) \big( v(F) - 2 \big) + 1.$$
\end{lemma}

We shall prove Lemma~\ref{extremal} using induction; in order to do so, we shall need to state a more general version of it (see Lemma~\ref{tech}, below). The statement is slightly technical, and we shall need some preparatory definitions. We shall use the following algorithm, which is simply a restatement of the Witness-Set Algorithm.

\begin{REA}
Let $G$ be a graph, let $r \ge 4$, and let $e \in \<G\>_{K_r} \setminus G$.
\begin{enumerate}
\item[1.] Run the Witness-Set Algorithm until edge $e$ is infected.
\item[2.] Let $(e_1,e_2,\ldots,e_m)$ be the infected edges which satisfy $F(e_j) \subset F(e)$ and $e_j \not\in G$, written in the order in which they are infected, where $e_m = e$.
\item[3.] For each $1 \le j \le m$, let $K^{(j)}$ be the $r$-clique which is completed by $e_j$.
\item[4.] Colour the edges $\{e_1,\ldots,e_m\}$ red, and note that $e_j \in K^{(j)} \setminus \big( K^{(1)} \cup \ldots \cup K^{(j-1)} \big)$.
\end{enumerate}
\end{REA}

The key observation is that $F(e) = \big( K^{(1)} \cup \ldots \cup K^{(m)} \big) \setminus \big\{ e_1, \ldots, e_m \big\}$, or, in words, $F(e)$ consists of all the non-red edges of the cliques which led to its infection. Indeed, the red edges were infected during the process, and so cannot be in $F(e)$; on the other hand, for each $1 \le j \le m$ the condition $F(e_j) \subset F(e)$ implies that $K^{(j)} \setminus \big\{ e_1, \ldots, e_m \big\} \subset F(f)$ for some $f \in K^{(m)}$. The reader should think about the Red Edge Algorithm in the following way: at each step an $r$-clique is added, and one of the \emph{new} edges of this clique is coloured red.

We shall bound the number of non-red edges after $t$ steps of the Red Edge Algorithm. Thus, given a realization of the algorithm and $t \in [m]$, define
$$B_t \, := \, \big( K^{(1)} \cup \ldots \cup K^{(t)} \big) \setminus \big\{ e_1, \ldots, e_t \big\}.$$
Note that $B_t \neq F(e_t)$ in general, since the condition $F(e_j) \subset F(e)$ for each $j \in [m]$ does not imply that $F(e_i) \subset F(e_t)$ for every $i \in [t]$. In order to state Lemma~\ref{tech}, we need to define two more parameters of the model, which will both play a key role in the induction step.

\begin{defn}
Let $\G_t$ denote the graph, obtained using the Red Edge Algorithm, whose vertices are the cliques $\big\{ K^{(1)}, \ldots, K^{(t)} \big\}$, and in which two cliques are adjacent if they share at least two vertices.

Let $\ell = \ell_t$ denote the number of components of $\G_t$, let $c(v) = c_t(v)$ denote the number of components of $\G_t$ containing the vertex $v \in V(G)$, and set
$$k \,=\, k_t \,=\, \ds\sum_{v \in V(B_t)} \big( c_t(v) - 1 \big).$$
Here, and throughout, we treat components of $\G_t$ as subsets of $V(G)$, and trust that this will not cause confusion.
\end{defn}

The following lemma implies Lemma~\ref{extremal}, since the graph $\G_m$ is connected (see below), and so $\ell_m = 1$ and hence $k_m = 0$.

\begin{lemma}\label{tech}
$e(B_t) \, \ge \,  \bigg( \ds\frac{{r \choose 2} - 2}{r-2} \bigg)\Big( v(B_t) + k - \ell r \Big)  + \ell \left( \ds{r \choose 2} - 1 \right).$
\end{lemma}

We shall prove Lemma~\ref{tech} by induction on $t$. The induction step will be relatively straightforward when $\ell_t \ge \ell_{t-1}$; when $\ell_t < \ell_{t-1}$ we shall need the following lemma.

Say that a (multi-)family of sets $\A$ is a double cover of $X$ if every element of $X$ is in at least two members of $\A$.

\begin{lemma}\label{odd2}
Let $m\ge 2$ and $r \ge 4$, and let $\A$ be a multi-family of subsets of $[m]$. If $\A$ is a double cover of $[m]$, and $|\A| \le r$, then
\begin{equation}\label{eq:odd2}
\left| \Big\{ \{A,B\} \in {\A \choose 2} \,:\, A \cap B \neq \emptyset \Big\} \right| \, \le \,  \lambda(r) \bigg( \sum_{A \in \A} |A| \,-\, 2m \bigg) \,+\, m.
\end{equation}
\end{lemma}

\begin{proof}[Proof of Lemma~\ref{odd2}]
We shall use induction on $m$. Suppose first that $m = 2$, and let $\A$ consist of $x$ sets of size two and $y$ sets of size one. If $x \ge 2$, then we have
$${x \choose 2} + xy + {y \choose 2} \,=\, {x+y \choose 2} \, = \,  \lambda(x+y) (x + y - 2) + 2 \,\le\, \lambda(r) \big( 2x + y - 4 \big) + 2,$$
since $2x + y \ge 4$ and $x + y \le r$. Similarly, if $x = 1$ then $y + {y-1 \choose 2} \le \lambda(r)(y-2) + 2$ for every $2 \le y \le r - 1$, and if $x = 0$ then $1 + {y-2 \choose 2} \le \lambda(r)(y-4) + 2$ for every $4 \le y \le r$.\footnote{In each case, note that it suffices to check the extreme values of $y$, and recall that $\lambda(r) \ge r/2$.}

So let $m \ge 3$, and let $\A$ be a multi-family as described, let $\T = \{ A \in \A : m \in A \}$, and apply the induction hypothesis to the multi-family $\A'$ obtained by removing $m$ from each element of $\T$. Letting $t = |\T|$, assume first that $t < r$. This gives
\begin{align*}
& \left| \Big\{ \{A,B\} \in {\A \choose 2} \,:\, A \cap B \neq \emptyset \Big\} \right| \, \le \,  \left| \Big\{ \{A,B\} \in {{\A'} \choose 2} \,:\, A \cap B \neq \emptyset \Big\} \right| \,+\, {t \choose 2} \\
& \hspace{3.3cm} \le \; \lambda(r) \left( \sum_{A \in \A'} |A| \,-\, 2m + 2 \right) \,+\, (m-1) \,+\, {t \choose 2} \\
& \hspace{3.3cm} = \; \lambda(r) \left( \sum_{A \in \A} |A| \,-\, 2m \right) \,+\, (m-1) \,+\, {t \choose 2} \,-\, \lambda(r) (t - 2),
\end{align*}
so it will suffice to show that $\lambda(r) (t - 2) \ge {t \choose 2} - 1$. But $\frac{{t \choose 2} - 1}{t-2} = \frac{t+1}{2}$, and $\lambda(r) \ge \frac{r}{2}$ if $r \ge 4$, so we are done unless $t = r$.

Finally, suppose that $t = r$. Then the left-hand side of~\eqref{eq:odd2} is equal to ${r \choose 2}$, and the right-hand side is at least
$$\lambda(r) \big( r \,+\, 2(m-1) \,-\, 2m \big) \,+\, m \, = \,  {r \choose 2} \,-\, 2 \,+\, m \, \ge \,  {r \choose 2},$$
since $\A$ is a double cover of $[m]$ and $m \ge 2$. The induction step, and hence the lemma, follows.
\end{proof}

In fact, the following reformulation of Lemma~\ref{odd2} will be more convenient for us in the proof below. Here $\N_0 = \{0,1,2,\ldots\}$, and $\P(m)$ denotes the non-empty subsets of $[m]$.

\begin{lemma}\label{odd1}
Let $m \ge 2$ and $r \ge 4$. Given any function $a : \P(m) \to \N_0$ such that $\sum_S a_S \le r$ and $\sum_{S \ni j} a_S \ge 2$ for every $j \in [m]$, we have
\begin{equation}\label{eq:odd1}
\sum_{S \in \P(m)} {a_S \choose 2} \,+\, \sum_{\{S, T\} \in \J} a_S a_T \, \le \,  \lambda(r) \bigg( \sum_{S \in \P(m)} a_S |S| \,-\, 2m \bigg) \,+\, m,
\end{equation}
where $\J = \big\{ \{S, T\} \in {{\P(m)} \choose 2} \, : \,S \cap T \neq \emptyset \big\}$.
\end{lemma}

\begin{proof}
We apply Lemma~\ref{odd2} to the multi-family $\A$ which contains exactly $a_S$ copies of $S$ for each $S \subset [m]$. The condition $\sum_{S \ni j} a_S \ge 2$ implies that $\A$ is a double cover, and $\sum_S a_S \le r$ implies that $|\A| \le r$. Thus~\eqref{eq:odd2} holds, which is clearly equivalent to~\eqref{eq:odd1}.
\end{proof}

We can now deduce Lemma~\ref{tech}.

\begin{proof}[Proof of Lemma~\ref{tech}]
We shall prove the lemma by induction on $t$. When $t = 1$ we have $v(B_1) = r$ and $e(B_1) = {r \choose 2} - 1$. Clearly $\ell_1 = 1$ and $k_1 = 0$, and
$$e(B_1) \, = \,  {r \choose 2} - 1 \, = \,   \lambda(r) \big( v(B_1) - r \big)  + \ds{r \choose 2} - 1,$$
so in fact equality holds. For the induction step we divide into three cases. Let $t \ge 2$, and assume that the lemma holds for smaller values of $t$.

\bigskip
\noindent \textbf{Case 1}: $\ell_t = \ell_{t-1} + 1$.
\medskip

Since $\G_t$ has one more component than $\G_{t-1}$, it follows that $K^{(t)}$ intersects every other clique in at most one vertex. Hence all of the edges of $K^{(t)}$ are new, and so
$$e(B_t) \, = \,  e(B_{t-1}) + {r \choose 2} - 1.$$
Now let $b$ be the number of vertices of $K^{(t)}$ which are not new, and hence intersect other components of $M$. Then $v(B_t) = v(B_{t-1}) + r - b$ and $k_t = k_{t-1} + b$, so, by the induction hypothesis for $t - 1$,
\begin{eqnarray*}
e(B_t) & \ge & \bigg( \ds\frac{{r \choose 2} - 2}{r-2} \bigg)\Big( v(B_{t-1}) + k_{t-1} - \ell_{t-1} r \Big)  + \big( \ell_{t-1} + 1 \big) \left( \ds{r \choose 2} - 1 \right)\\
& = & \bigg( \ds\frac{{r \choose 2} - 2}{r-2} \bigg)\Big( v(B_t) + k_t - \ell_t r \Big)  + \ell_t \left( \ds{r \choose 2} - 1 \right)
\end{eqnarray*}
as required.

\bigskip
\noindent \textbf{Case 2}: $\ell_t = \ell_{t-1}$.
\medskip

Since $\G_t$ and $\G_{t-1}$ have the same number of components, it follows that $K^{(t)}$ must intersect some component, $C_1$, in at least two vertices, and intersects every clique not in $C_1$ in at most one vertex. Thus, the only edges of $K^{(t)}$ which are not new have both endpoints in $C_1$. Hence, letting $a = |K^{(t)} \cap C_1|$, we have
$$e(B_t) \, \ge \,  e(B_{t-1}) + {r \choose 2} - {a \choose 2} - 1.$$
Now, let $b$ be the number of vertices of $K^{(t)} \setminus C_1$ which are not new, and hence intersect other components of $\G_t$. Then $v(B_t) = v(B_{t-1}) + r - a - b$ and $k_t = k_{t-1} + b$, so, by the induction hypothesis for $t - 1$,
\begin{eqnarray*}
e(B_t) & \ge & \bigg( \ds\frac{{r \choose 2} - 2}{r-2} \bigg)\Big( v(B_{t-1}) + k_{t-1} - \ell_{t-1} r \Big)  + \big( \ell_{t-1} + 1 \big) \left( \ds{r \choose 2} - 1 \right) - {a \choose 2}\\
& \ge & \bigg( \ds\frac{{r \choose 2} - 2}{r-2} \bigg)\Big( v(B_t) + k_t - \ell_t r - r + a \Big)  + \big( \ell_t + 1 \big)  \left( \ds{r \choose 2} - 1 \right)- {a \choose 2}.
\end{eqnarray*}
If $a \le r - 1$ then
$$\ds{r \choose 2} - 1 - {a \choose 2} - (r-a) \bigg( \ds\frac{{r \choose 2} - 2}{r-2} \bigg) \, \ge \,  0,$$
since the worst cases are the extremes ($a = 2$ and $a = r - 1$), and using the fact that $r \ge 4$. But if $a = r$, then our bound on $e(B_t)$ can be improved to $e(B_t) \ge e(B_{t-1})$ (which is trivial since we are not allowed to colour edges of $B_{t-1}$ red), and $v(B_t) = v(B_{t-1})$, so we are done in this case as well.

\bigskip
\noindent \textbf{Case 3}: $\ell_t < \ell_{t-1}$.
\medskip

This case is more difficult: to prove it, we shall use Lemma~\ref{odd1}. Set $m = \ell_{t-1} - \ell_t + 1$, and observe that $m \ge 2$, and that $K^{(t)}$ intersects $m$ components $C_1,\ldots,C_m$ in at least two vertices each, and intersects every clique not in these components in at most one vertex. Define, for each $S \in \P(m)$,
$$a_S \, = \, \big| \big\{ v \in K^{(t)} \,:\, v \in C_j \,\Leftrightarrow\, j \in S \big\} \big|,$$
and note that $\sum_{S} a_S \le r$ and $\sum_{S \ni j} a_S = |K^{(t)} \cap C_j| \ge 2$. Moreover, set
$$e(A) \, = \, \sum_{S \in \P(m)} {a_S \choose 2} + \sum_{\{S,T\} \in \J} a_S a_T,$$
where $\J = \big\{ \{S, T\} \in {{\P(m)} \choose 2} \, : \,S \cap T \neq \emptyset \big\}$, as in Lemma~\ref{odd1}. We claim that
$$e(B_t) \, \ge \, e(B_{t-1}) + {r \choose 2} - e(A) - 1.$$
Indeed, if an edge of $K^{(t)}$ was already present in $B_{t-1}$, then there must be a clique (and hence a component $C_j$) which contains both of its endpoints. Moreover, $e(A)$ counts exactly the number of pairs of vertices of $K^{(t)}$ which are both in some component $C_j$.

Let $a = \sum_{S} a_S$ denote the number of vertices of $K^{(t)} \cap \big( C_1 \cup \ldots \cup C_m \big)$, and let $b$ denote the number of vertices of $K^{(t)} \setminus \big( C_1 \cup \ldots \cup C_m \big)$ which intersect other components of $\G_t$. Then $v(B_t) = v(B_{t-1}) + r - a - b$, and recall that $\ell_t  = \ell_{t-1}  - (m - 1)$. Also, let
$$c \, = \, \sum_{S \in \P(m)} a_S \big( |S| - 1 \big),$$
and observe that $k_t \le k_{t-1} + b - c$, since $K^{(t)}$ unifies the components $C_1,\ldots,C_m$, and so if $\{ j \in [m] : v \in C_j \} = S$ then $c_t(v) = c_{t-1}(v) - |S| + 1$.

Thus, by the induction hypothesis for $t - 1$,
\begin{eqnarray*}
e(B_t) & \ge & \bigg( \ds\frac{{r \choose 2} - 2}{r-2} \bigg)\Big( v(B_{t-1}) + k_{t-1} - \ell_{t-1} r \Big)  + \big( \ell_{t-1} + 1 \big) \left( \ds{r \choose 2} - 1 \right) - e(A)\\
& \ge & \bigg( \ds\frac{{r \choose 2} - 2}{r-2} \bigg)\Big( v(B_t) + k_t - \ell_t r - mr + a + c \Big)  + \big( \ell_t + m \big)  \left( \ds{r \choose 2} - 1 \right) - e(A).
\end{eqnarray*}
Recall that $a \le r$ and $\sum_{S \ni j} a_S = |K^{(t)} \cap C_j| \ge 2$. Hence, by Lemma~\ref{odd1},
$$e(A) \, \le \,  \lambda(r) \bigg( \sum_{S \in \P(m)} a_S |S| \,-\, 2m \bigg) + m \, = \,   \lambda(r) \big( a + c \big) \,-\, m \left( \lambda(r) r - {r \choose 2} + 1\right),$$
since $2\lambda(r) - 1 = \lambda(r) r - {r \choose 2} + 1$. Thus
$$e(B_t) \, \ge \,  \bigg( \ds\frac{{r \choose 2} - 2}{r-2} \bigg)\Big( v(B_t) + k_t - \ell_t r \Big)  + \ell_t  \left( \ds{r \choose 2} - 1 \right) ,$$
as required. This completes the induction step, and hence the proof of the lemma.
\end{proof}

For completeness, let us quickly note formally that Lemma~\ref{extremal} follows immediately from Lemma~\ref{tech}.

\begin{proof}[Proof of Lemma~\ref{extremal}]
Let $F$ be an $r$-witness set for the graph $G$ and the edge $e$, and run the Red Edge Algorithm. We claim that the graph $\G_m$ is connected. To see this, first observe that if $f \in F(e)$, then there must be a path in $\G_m$ from $K^{(m)}$ to a clique $K^{(t)}$ containing $f$; indeed, this follows immediately from the definition of the Witness Set Algorithm, working backwards from $e$. Now let $j \in [m]$, and note that since $F(e_j)$ is non-empty and $F(e_j) \subset F(e)$, there exists an edge $f \in F(e_j) \cap F(e)$. By the previous observation, it follows there is a path in $\G_m$ from $K^{(m)}$ to a clique $K^{(t)}$ containing $f$, and a path in $\G_m$ from $K^{(j)}$ to a clique $K^{(t')}$ containing $f$. Since $f \in K^{(t)} \cap K^{(t')}$, these cliques are neighbours in $\G_m$, and hence there is a path from $K^{(j)}$ to $K^{(m)}$ in $\G_m$ for any $j \in [m]$, as claimed.

It follows that $\ell_m = 1$, and so $c_m(v) = 1$ for every $v \in V(B_m)$, which means that $k_m = 0$. Hence, by Lemma~\ref{tech},
$$e(F) \, \ge \,  \bigg( \ds\frac{{r \choose 2} - 2}{r-2} \bigg)\Big( v(F) - r \Big)  + \left( \ds{r \choose 2} - 1 \right) \, = \,  \lambda(r) \big( v(F) - 2 \big) + 1,$$
as required.
\end{proof}

\subsection{Bootstrap methods}

To deduce Proposition~\ref{prop:lower}, we shall borrow a simple but important idea from the theory of bootstrap percolation. The following lemma is based on an idea of Aizenman and Lebowitz~\cite{AL}.

\begin{lemma}\label{ALlem}
Let $F$ be an $r$-witness set on a graph $G$, and let $L \in \N$. If $e(F) \ge L$, then there exists an edge $f \in E(G)$ with
$$L \, \le \,  e\big( F(f) \big) \, \le \,  {r \choose 2}L$$
in the same realization of the Witness-Set Algorithm.
\end{lemma}

\begin{proof}
Run the Witness-Set Algorithm, and observe that the maximum size of $e(F(f))$, over all infected edges, increases by at most a factor of ${r \choose 2}$ at each step of the process. It follows immediately that a graph $F(f)$ as described must have been created, at some point in the process. Moreover, such a graph exists with $F(f) \subset F$.
\end{proof}

We have finally finished with our deterministic preliminaries, and it is time to reintroduce randomness. There is, however, little left to do: the bound we require will follow easily from Lemmas~\ref{extremal} and~\ref{ALlem} by Markov's inequality.

\medskip
For each $m \in \N$ and every $e \in E(K_n)$, let
$$Y_m(e) \,:=\, \left| \Big\{ F \subset G_{n,p} \,:\, e \subset V(F), \text{ and } e(F) = m \ge \lambda(r) \big( v(F) - 2 \big) + 1 \Big\} \right|$$
be the random variable which counts the number of subgraphs $F$ of $G_{n,p}$ whose vertex set contains the endpoints of the edge $e$, and have $m \ge \lambda(r) \big( v(F) - 2 \big) + 1$ edges. We first bound the expected size of $Y_m(e)$.

\begin{lemma}\label{expY}
For every $r \ge 4$, there exists a $C(r) > 0$ such that the following holds. If $n \in \N$ and $p > 0$ satisfy $p n^{1/\lambda(r)}\log n \le 1/(2e)$, and $n$ is sufficiently large, then
$$\Ex \big( Y_m(e) \big) \, \le \,  \bigg( \frac{m + C(r)}{2 {r \choose 2} \log n} \bigg)^{m - \lambda(r)}$$
for every $e \in E(K_n)$ and every $\lambda(r) + 1 \le m \le {r \choose 2} \log n$.
\end{lemma}

\begin{proof}
Let $\ell \in \N$ be maximal such that $m \ge \lambda(r) \big( \ell - 2 \big) + 1$. Then $v(F) \le \ell$, and hence
$$\Ex\big( Y_m(e) \big) \, \le \, \bigg( \sum_{j = 0}^{\ell - 2} {n \choose j} \bigg) {{\ell^2/2} \choose m} p^m \, \le \,  \frac{2 \cdot n^{\ell-2}}{(\ell-2)!} \left( \frac{e p \ell^2}{2m} \right)^{m} \, \ll \, \left( n^{1/\lambda(r)} \cdot \frac{e p \ell^2}{2 m} \right)^{\lambda(r)( \ell - 2 )}$$
where the last inequality follows since $m - \lambda(r) \big( \ell - 2 \big) \ge 1$ and $p \ell^2 \le n^{-1/2\lambda(r)} = o(1)$. Next, observe that
$$\frac{\ell^2}{2 m} \, \le \, \frac{m + C(r)}{{r \choose 2}}$$
if $C(r)$ is sufficiently large, since $m\big( m + C(r) \big) \ge \big( m + 2\lambda(r) \big)^2 \ge \big( \lambda(r) \ell \big)^2 \ge \frac{r^2\ell^2}{4}$. Hence, recalling that $\lambda(r) \big( \ell - 2 \big) \ge m - \lambda(r)$ and $m \le {r \choose 2} \log n$, we obtain
$$\Ex\big( Y_m(e) \big) \, \le \, \left( n^{1/\lambda(r)} e p \cdot \frac{m + C(r)}{{r \choose 2}} \right)^{\lambda(r)( \ell - 2 )}  \, \le \, \left(\frac{m + C(r)}{2{r \choose 2}\log n} \right)^{m - \lambda(r)}$$
if $n$ is sufficiently large, as required.
\end{proof}

We can now easily deduce Proposition~\ref{prop:lower}.

\begin{proof}[Proof of Proposition~\ref{prop:lower}]
Let $r \ge 4$, let $n \in \N$, and let $p = p(n) > 0$ satisfy $p n^{1/\lambda(r)} \log n \le 1/(2e)$. We claim that, for every $e \in E(K_n)$,
$$\Pr\Big( e \in \< G_{n,p} \>_{K_r} \Big) \,\to\, 0$$
as $n \to \infty$. Indeed, suppose that $e \in \< G_{n,p} \>_{K_r}$, run the Witness-Set Algorithm, and consider the graph $F = F(e) \subset G_{n,p}$.

Suppose first that $e(F) \le \log n$.  By Lemma~\ref{extremal}, we have
$$e(F) \, \ge \, \lambda(r) \big( v(F) - 2 \big) + 1,$$
and thus either $e \in G_{n,p}$, or $Y_m(e) \ge 1$ for some
$$\lambda(r)(r - 2) + 1 \, \le \, m \, \le \, \log n.$$
By Lemma~\ref{expY}, this has probability at most
$$p \,+ \sum_{m = \lambda(r)(r-2) + 1}^{\log n} \Ex\big( Y_m(e) \big) \, \le \,  p \,+ \sum_{m = \lambda(r)(r-2) + 1}^{\log n} \bigg( \frac{m + C(r)}{2 {r \choose 2} \log n} \bigg)^{m - \lambda(r)} \to \; 0,$$
as $n \to \infty$, as claimed.

So suppose next that $e(F) \ge \log n$. By Lemma~\ref{ALlem}, there must exist an edge $f$ in $K_n$ such that $\log n \le e(F(f)) \le {r \choose 2} \log n$, which means that $Y_m(f) \ge 1$ for some $\log n \le m \le {r \choose 2} \log n$. By Lemma~\ref{expY}, the expected number of such edges $f$ is at most
$${n \choose 2} \sum_{m = \log n}^{{r \choose 2} \log n} \bigg( \frac{m + C(r)}{2 {r \choose 2} \log n} \bigg)^{m - \lambda(r)} \le \; n^2 \bigg( \frac{2}{r^2} \bigg)^{\log n - \lambda(r)} \to \; 0,$$
as $n \to \infty$, since $r^2/2 \ge 8 > e^2$. This proves the proposition.
\end{proof}

We finish by noting that Theorem~\ref{thm:Kr} follows immediately from Propositions~\ref{prop:bal} and~\ref{prop:lower}.

\begin{proof}[Proof of Theorem~\ref{thm:Kr}]
By Proposition~\ref{prop:bal}, we have
$$p_c(n,H) \, \ll \,  \big( \log n \big)^{2 / \lambda(H)} n^{-1 / \lambda(H)}$$
for every balanced graph $H$. Moreover $K_r$ is balanced, since
$$\frac{{r-1 \choose 2} - 1}{r-3} \, \le \, \frac{{r \choose 2} - 2}{r-2}$$
for every $r \ge 4$, and $\lambda(K_r) = \lambda(r) \ge 2$, so the upper bound follows.

For the lower bound, suppose that $pn^{1/\lambda(r)} \log n \le 1/(2e)$, and $n$ is sufficiently large. By Proposition~\ref{prop:lower}, we have
$$\Pr\Big( e \in \< G_{n,p} \>_{K_r} \Big) \,\to\, 0$$
for every edge $e \in E(K_n)$. Thus $G_{n,p}$ does not percolate, with high probability, as required.
\end{proof}

\section{The threshold for $K_4$-percolation}\label{k4sec}

In this section we shall prove Theorem~\ref{k=4}, which determines the threshold for $K_4$-percolation on $K_n$. The proof is quite different from that of Theorem~\ref{thm:Kr}, and uses ideas from the study of $2$-neighbour bootstrap percolation on $[n]^d$ (see~\cite{AL,BB}, or the more recent improvements in~\cite{BBMhigh,GHM,Hol}).

We begin with a simple but key observation. A collection $\K$ of cliques is said to be \emph{triangle-free} if there do not exist distinct vertices $u,v,w$ and cliques $A,B,C \in \K$ such that $u \in V(A) \cap V(B)$, $v \in V(B) \cap V(C)$ and $w \in V(A) \cap V(C)$.

\begin{obs}\label{obsk4}
For every graph $G$, the graph $\< G\>_{K_4}$ consists of a triangle-free collection of edge-disjoint cliques.
\end{obs}

\begin{proof}
To show that $\< G\>_{K_4}$ is a collection of edge-disjoint cliques, simply note that if two cliques $A$ and $B$ share more than one vertex, then the closure $\< A \cup B \>_{K_4}$ is a clique on vertex set $V(A) \cup V(B)$. To prove that this collection is triangle-free, observe that if $A$, $B$ and $C$ form a triangle, then the closure $\< A \cup B \cup C \>_{K_4}$ is a clique on vertex set $V(A) \cup V(B) \cup V(C)$.
\end{proof}

Say that a clique $K$ is \emph{internally spanned} by a graph $G$ if $\< G \cap K \>_{K_4} = K$. We shall study, for each $\ell \in \N$ and $p > 0$, the probability
$$P(\ell,p) \,:=\, \Pr\Big( K_\ell \textup{ is internally spanned by }G_{n,p} \Big).$$
In order to do so, we shall introduce a simple algorithm for filling $K_n$, which we call the Clique-Process. It is analogous to the `rectangle process' in two-neighbour bootstrap percolation on $[n]^d$ (see Proposition~30 of~\cite{Hol} or Theorem~11 of~\cite{BB}).

\begin{cliquealg}
Let $G$ be a graph on $n$ vertices, and run the $K_4$-process as follows:
\begin{enumerate}
\item[0.] At each step of the process, we will maintain a collection $(R_1,A_1),\ldots,(R_m,A_m)$, where $R_j$ is a clique and $A_j \subset E(G)$, such that $\< A_j \>_{K_4} = R_j$ for each $j \in [m]$.\\[-2ex]
\item[1.] At time zero, set $R_j = A_j = \{e_j\}$ for each $j \in [m]$, where $E(G) = \{e_1,\ldots,e_m\}$. \\[-2ex]
\item[2.] At time $t \in 2\ZZ$, choose a pair $\{i,j\}$ such that $|R_i \cap R_j| \ge 2$, if such a pair exists. Delete $(R_i,A_i)$ and $(R_j,A_j)$, and replace them with $(\< A_i \cup A_j \>_{K_4} , A_i \cup A_j)$. \\[-2ex]
\item[3.] At time $t \in 2\ZZ + 1$, choose a triple $\{i,j,k\}$ such that $R_i$, $R_j$, and $R_k$ form a triangle in the hypergraph defined by the cliques, if such a triple exists. Delete $(R_i,A_i)$, $(R_j,A_j)$ and $(R_k,A_k)$, and replace them with $(\< A_i \cup A_j \cup A_k \>_{K_4} , A_i \cup A_j \cup A_k)$. \\[-2ex]
\item[4.] Repeat steps $2$ and $3$ until the collection $(R_1,A_1),\ldots,(R_m,A_m)$ stabilizes, that is, until there are no more pairs as in step 2, or triples as in step 3.
\end{enumerate}
\end{cliquealg}

The algorithm terminates by the proof of Observation~\ref{obsk4}. Observe moreover that the $A_j$ are in fact \emph{disjoint} sets of edges of $G$.
We can now easily deduce the following bound, which was first proved by Bollob\'as~\cite{Bela68}.

\begin{lemma}\label{2l-3}
If $G$ internally spans $K_\ell$ then $e(G) \ge 2\ell - 3$.
\end{lemma}

\begin{proof}
We shall use induction on $\ell$; for $\ell \le 3$ the result is trivial. Now suppose that $G$ internally spans $R = K_\ell$, and run the Clique Process for $G$. At the penultimate step we have either two or three disjointly internally spanned proper sub-cliques of $R$, which together span $R$. If these cliques are $S = \< A \>_{K_4}$ and $T = \< B \>_{K_4}$, then
$$e(G) \, \ge \, e(A) + e(B) \, \ge \, 2\big( v(S) + v(T) \big) - 6 \, \ge \, 2\ell - 2,$$
since $A \cap B = \emptyset$ and $|S \cap T| \ge 2$, so $v(S) + v(T) \ge \ell + 2$. If they are $S = \< A \>_{K_4}$, $T = \< B \>_{K_4}$ and $U = \< C \>_{K_4}$, then we have
$$e(G) \, \ge \, e(A) + e(B) + e(C) \, \ge \, 2\big( v(S) + v(T) + v(U) \big) - 9 \, \ge \, 2\ell - 3,$$
since $A,B,C$ are pairwise disjoint and $(S,T,U)$ form a triangle, so $v(S) + v(T) + v(U) \ge \ell + 3$.
\end{proof}

The following bounds on $P(\ell,p)$ now follow easily.

\begin{lemma}\label{Plp}
For every $3 \le \ell \in \N$ and $p \in (0,1)$ with $p \ell^2 \le 1$,
$$\left( \frac{1}{2e^2} \right)^{\ell} (\ell p )^{2\ell-3} \, \le \,  P(\ell,p) \, \le \,  4^3 \bigg( \frac{e}{4} \bigg)^{2\ell} \big(\ell p \big)^{2\ell-3}.$$
\end{lemma}

\begin{proof}
For the lower bound, simply count the graphs on vertex set $[\ell]$, and with $2\ell - 3$ edges, in which every vertex $j \ge 3$ sends exactly two edges `backwards' in the order induced by $\ZZ$. It is easy to see, by induction on $t$, that the clique $K_t$ with vertex set $[t]$ is internally spanned, for each $t \in [\ell]$. The number of such graphs is
$$\prod_{j = 3}^\ell {j-1 \choose 2} \, \ge \,  \frac{(\ell!)^2}{2^\ell \ell^3} \, \ge \,  \frac{2\pi \ell^{2\ell-3}}{(2e^2)^\ell},$$
by Stirling's formula, and each is an \emph{induced} subgraph of $G_{n,p}$ with probability at least $p^{2\ell-3} (1 - p)^{\ell^2} \ge p^{2\ell-3}/2\pi$, where the bound $(1 - p)^{\ell^2} \ge e^{-3/2} > 1 / 2\pi$ follows since $p\ell^2 \le 1$ and $\ell \ge 3$. Since these events are mutually exclusive, the lower bound follows.

For the upper bound, recall that, by Lemma~\ref{2l-3}, if a graph $G$ internally spans $K_\ell$ then $e(G) \ge 2\ell - 3$. Since $\big( \frac{4\ell}{4\ell - 6} \big)^{2\ell - 3} \le e^3$, it follows that
$$P(\ell,p) \, \le \,  {\ell^2/2 \choose {2\ell - 3}} p^{2\ell-3} \, \le \,  \left( \frac{ e \ell }{ 4\ell - 6} \right)^{2\ell-3} \big(\ell p \big)^{2\ell-3} \, \le \,  4^3 \bigg( \frac{e}{4} \bigg)^{2\ell} \big(\ell p \big)^{2\ell-3},$$
as required.
\end{proof}

\subsection{The lower bound}

The following lemma, like Lemma~\ref{ALlem}, it is based on an idea of Aizenman and Lebowitz~\cite{AL}, who proved the corresponding result in the context of two-neighbour bootstrap percolation on $[n]^d$. The lower bound in Theorem~\ref{k=4} will follow by combining it with Lemma~\ref{Plp}.

\begin{lemma}\label{ALlemma}
Suppose that $\< G \>_{K_4} = K_n$, and let $1 \le L \le n$. There exists a clique $K \subset K_n$ which is internally spanned by $G$, with
$$L \, \le \, v(K) \, \le \, 3L.$$
\end{lemma}

\begin{proof}[Proof of Lemma~\ref{ALlemma}]
Suppose that $\< G \>_{K_4} = K_n$, and run the Clique Process for $G$. At each step of the process, the value of $\max_{j \in [m]} v(R_j)$ increases by a factor of at most three. Hence, for every $L \in [n]$, there exists a clique $K \subset K_n$ with
$$L \, \le \, v(K) \, \le \, 3L,$$
which is internally spanned by $G$, as claimed.
\end{proof}

We remark that this result does not generalize to $K_r$-percolation for $r \ge 5$. In fact, it is not hard to construct a graph $G$ for which $\< G \>_{K_r} = K_n$, but no clique $K_\ell$ with $r < \ell < n$ is internally spanned.

We can now prove the lower bound on $p_c(n,K_4)$ in Theorem~\ref{k=4}. It follows easily from Lemmas~\ref{Plp} and~\ref{ALlemma}, using Markov's inequality.

\begin{prop}\label{k4:lower}
If $p^2 n \log n \le 16/e^5$, then
$$\Pr\Big( \< G_{n,p} \>_{K_4} = K_n \Big) \, \to \, 0$$
as $n \to \infty$.
\end{prop}

\begin{proof}
Let $p^2 n \log n = 16/e^5$ and $L = \log n$. By Lemma~\ref{ALlemma}, if $\< G_{n,p} \>_{K_4} = K_n$ then there exists an internally spanned clique $R$ with $L \le v(R) \le 3L$. By Lemma~\ref{Plp}, the expected number of such cliques is at most
$$4^3 \sum_{\ell = L}^{3L} {n \choose \ell} \bigg( \frac{e}{4} \bigg)^{2\ell} \big(\ell p \big)^{2\ell-3} \, \le \, \sum_{\ell = L}^{3L} \left( \frac{4}{\ell p} \right)^3 \left( \frac{en}{\ell} \cdot \frac{e^2}{16} \cdot \ell^2 p^2 \right)^\ell \, \le \, \sum_{\ell = L}^{3L} n^{3/2} \left( \frac{\ell}{e^2 \log n} \right)^{\ell},$$
since $\ell p \gg n^{-1/2}$. Thus
$$\sum_{\ell = L}^{3L} n^{3/2} \left( \frac{\ell}{e^2 \log n} \right)^{\ell} \, \le \,  3L \cdot  n^{3/2} e^{-2L} \, \to \,  0$$
as $n \to \infty$, as required.
\end{proof}

\subsection{The upper bound}

We shall use the second moment method (and Lemma~\ref{Plp}) in order to show that $G_{n,p}$ internally spans a clique of order $\sim \log n$ with high probability. We will then deduce the upper bound in Theorem~\ref{k=4} using sprinkling.

Let $X(\ell,p)$ denote the random variable which counts the number of copies of $K_\ell$ which are internally spanned by $G_{n,p}$. We first bound the expected value of $X(\ell,p)$.

\begin{lemma}\label{exX}
For every $n \in \N$, $3 \le \ell \in \N$ and $p \in (0,1)$ with $p \ell^2 \le 1$,
$$\Ex\big( X(\ell,p) \big) \, \ge \,  \left( \frac{p^2 n \ell}{2e^2}  \right)^{\ell} \left( \frac{1}{\ell p} \right)^3.$$
\end{lemma}

\begin{proof}
By Lemma~\ref{Plp}, we have
\begin{eqnarray*}
\Ex\big( X(\ell,p) \big) \, \ge \,  {n \choose \ell} \left( \frac{1}{2e^2} \right)^{\ell} (\ell p )^{2\ell-3}  \, \ge \,   \left( \frac{n}{\ell} \cdot \frac{1}{2e^2} \cdot \ell^2 p^2 \right)^{\ell} \left( \frac{1}{\ell p} \right)^3
\end{eqnarray*}
as required.
\end{proof}

To bound the variance of $X(\ell,p)$, we shall use the following extension of Lemma~\ref{2l-3}. Given cliques $S \subset R$, let
$$D(S,R) \,:= \, \Big\{ \big\< (G_{n,p} \cup S) \cap R \big\>_{K_4} = R \Big\}$$
denote the event that $R$ is internally spanned by $G_{n,p} \cup S$. Lemma~\ref{2l-3} is equivalent to the case $v(S) = 3$ of the following lemma.

\begin{lemma}\label{Dext}
If $D(S,R)$ holds, then $e\big( (G_{n,p} \setminus S) \cap R  \big) \,\ge\, 2\big( v(R) - v(S) \big)$.
\end{lemma}

\begin{proof}
We shall use induction on $\ell = v(R)$. Suppose that $D(S,R)$ holds, and apply the Clique Process, except starting with the clique $S$ already formed. Suppose at the penultimate step we have two disjointly internally spanned cliques, $T = \< A \cup S \>_{K_4}$ and $U = \< B \>_{K_4}$, where $A \cap S = A \cap B = \emptyset$. (That $A$ and $B$ may be taken to be disjoint follows by the comment after the Clique Process.) By the induction hypothesis, we have
$$e\big( (G_{n,p} \setminus S) \cap R  \big) \, \ge \, e(A) + e(B) \, \ge \, 2 \big( v(T) - v(S) \big) + 2 v(U) - 3 \, > \, 2\big( v(R) - v(S) \big),$$
since $|T \cap U| \ge 2$. The case of three cliques $T = \< A \cup S \>_{K_4}$, $U = \< B \>_{K_4}$ and $W = \< C \>_{K_4}$ is similar; we obtain
\begin{eqnarray*}
e\big( (G_{n,p} \setminus S) \cap R  \big) & \ge & e(A) + e(B) + e(C) \\
& \ge & 2 \big( v(T) - v(S) \big) + 2 \big( v(U) + v(W) \big) - 6 \, \ge \, 2\big( v(R) - v(S) \big),
\end{eqnarray*}
as required, since $T$, $U$ and $W$ form a triangle, so $v(T) + v(U) + v(W) \ge v(R) + 3$.
\end{proof}

We can now bound the variance of $X(\ell,p)$. Let $P(k,\ell) = \Pr\big( D(K_k,K_\ell) \big)$.

\begin{lemma}\label{var}
Let $n \in \N$, $4\ell \le \log n$, $p \ell^2 \le 1$ and $p^2 n \log n \ge 33/e$. Then
$$\Var\big( X(\ell,p) \big) \, \ll \,  \Ex\big( X(\ell,p) \big)^2$$
as $n \to \infty$.
\end{lemma}

\begin{proof}
We first claim that
$$\Var\big( X(\ell,p) \big) \, \le \,  \sum_{k = 2}^\ell \Ex\big( X(\ell,p) \big) {\ell \choose k} {n \choose {\ell-k}} P(k,\ell).$$
This follows by considering ordered pairs $(S,T)$ of internally spanned $\ell$-cliques which intersect in a $k$-clique. By Lemma~\ref{Dext}, if $D(S \cap T,T)$ holds then there are at least $2(\ell - k)$ edges of $G_{n,p}$ in $T \setminus S$, and so
$$P(k,\ell) \, \le \,  {(\ell^2 - k^2)/2 \choose {2\ell - 2k}} p^{2(\ell-k)} \, \le \,  \left( \frac{e(\ell+k) p}{4} \right)^{2(\ell - k)}.$$
Thus, by Lemma~\ref{exX},
$$P(k,\ell) \, \le \,  \left( \frac{e(\ell+k) p}{4} \right)^{2(\ell - k)} \left( \frac{2e^2}{p^2 n \ell}  \right)^{\ell} \left( \ell p \right)^3 \Ex\big( X(\ell,p) \big).$$
Now, using the facts that $k \le \ell$ and $(\ell+k)^{\ell-k} \le e^{2k} (\ell-k)^{\ell-k}$, an easy calculation gives that
$${\ell \choose k} {n \choose {\ell-k}} \left( \frac{e(\ell+k) p}{4} \right)^{2(\ell - k)} \left( \frac{2e^2}{p^2 n \ell} \right)^\ell$$
is at most
$$\left( en \cdot \frac{e^2 (\ell + k) p^2}{16} \cdot \frac{2e^2}{p^2 n \ell} \right)^\ell \left( e^2 \cdot \frac{e\ell}{k} \cdot \frac{1}{en} \cdot \frac{16}{e^2 (\ell + k) p^2} \right)^k \, \le \, \left( \frac{e^5}{4} \right)^\ell  \left( \frac{16}{k p^2 n} \right)^k,$$
and hence
$$\Var\big( X(\ell,p) \big) \, \le \, \Ex\big( X(\ell,p) \big)^2  \left( \ell p \right)^3 \left( \frac{e^5}{4} \right)^\ell  \sum_{k = 2}^\ell  \left( \frac{16}{k p^2 n} \right)^k.$$
Finally, recall that $4\ell \le \log n$ and $p^2 n \log n \ge 33/e$, and observe that therefore $\left( \ell p \right)^3 \big( \frac{e^5}{4} \big)^\ell \ll 1 / \sqrt{n}$. Thus, using the fact that $(1/Cx)^x \le e^{1/Ce}$, we obtain
$$\Var\big( X(\ell,p) \big) \, \le \, \Ex\big( X(\ell,p) \big)^2 \cdot \frac{\ell}{\sqrt{n}} \exp\left( \frac{16}{e p^2 n} \right) \, \ll \, \Ex\big( X(\ell,p) \big)^2,$$
as required.
\end{proof}

Using Chebyshev's inequality, and sprinkling, we can now deduce the following result.

\begin{prop}\label{k4:upper}
If $p^2 n \log n \ge (24)^2$, then
$$\Pr\Big( \< G_{n,p} \>_{K_4} = K_n \Big) \, \to \, 1$$
as $n \to \infty$.
\end{prop}

\begin{proof}
Set $4\ell = \log n$ and $p^2 n \log n = 16 > 33/e$, and observe that the conditions of Lemmas~\ref{exX} and~\ref{var} are satisfied. By Lemma~\ref{exX}, we obtain
$$\Ex\big( X(\ell,p) \big) \, \ge \, \left( \frac{p^2 n \ell}{2e^2}  \right)^{\ell} \left( \frac{1}{\ell p} \right)^3 \, \ge \, \left( \frac{2}{e^2}  \right)^{\log n / 4} n^{3/2 - o(1)} \, \to \, \infty$$
as $n \to \infty$. Thus, by Lemma~\ref{var} and Chebyshev's inequality, with high probability there exists a copy of $K_\ell$ which is internally spanned by $G_{n,p}$.

Now let $G_0 = G_{n,p}$ be a random graph with density $p$, and for each $j \in \N$ set $p_j = 2^{-j+2} p$ and let $G_j = G_{n,p_j}$ be a random graph with density $p_j$, chosen independently of all others. We make the following claim.

\medskip
\noindent \textbf{Claim 1:} There exists an $\eps > 0$ such that the following holds. If $t := 2^{2j-2} \ell \le \eps n$, then $\< G_j \cup K_t \>_{K_4}$ contains a clique of size $4t$ with probability at least $1 - e^{-t/8}$.

\begin{proof}[Proof of Claim 1]
Observe that every vertex $v$ that has at least two neighbours in $K_t$ (in the graph $G_j$) is added to the clique in $\< G_j \cup K_t \>_{K_4}$. It therefore suffices to show that there are at least $3t$ such vertices, with high probability. The expected number of such vertices is at least
$$\frac{3n}{4} \cdot {t \choose 2} p_j^2 (1 - p_j)^{t-2} \, \ge \, t p^2 n \ell \, = \, 4t$$
since $p_j t = 2^{j} \ell p = 2p \sqrt{\ell t} = O(\sqrt{\eps})$.

This event (having two neighbours) is independent for each vertex. Thus, by Chernoff's inequality, with probability at least $1 - e^{-t/8}$, the number of such vertices is at least $3t$, as required.
\end{proof}

We apply the claim for each $j \ge 0$. It follows that, with high probability, $\< G_0 \cup \bigcup_{j = 1}^\infty G_j \>_{K_4}$ contains a clique of order $\eps n$, for some $\eps > 0$. Finally, let $G_0'$ be another independent copy of $G_{n,p}$.

\medskip
\noindent \textbf{Claim 2:} For every $\eps > 0$, if $t \ge \eps n$ then $\< G_{n,p} \cup K_t \>_{K_4} = K_n$ with high probability.

\begin{proof}[Proof of Claim 2]
We apply the same argument as in the proof of Claim~1. Indeed, the probability that a vertex $v$ has at most one neighbour in $K_t$ is at most
$$(1 - p)^t + tp(1-p)^{t-1} \,\le\, \big( 1 + 2tp \big) e^{-tp} \,\ll\, \frac{1}{n^2}$$
as $n \to \infty$. Hence, by Markov's inequality, the probability that there exists such a vertex is at most $1/n$, as required.
\end{proof}

To complete the proof, we simply note that the graph $G = G_0 \cup G_0' \cup \bigcup_{j = 1}^\infty G_j$ is a random graph $G_{n,p^*}$ of density
$$p^* \,\le\, 2p + \sum_{j = 1}^\infty 2^{-j+2} p \, = \, 6p,$$
and $G_{n,p^*}$ percolates in the $K_4$-process with high probability, as required.
\end{proof}

Theorem~\ref{k=4} follows immediately from Propositions~\ref{k4:lower} and~\ref{k4:upper}.

\section{Other graphs, and open problems}\label{othersec}

In this section we shall mention some simple results for graphs other than $K_r$, and state several of the many open problems relating to this model. Since the results will all be fairly straightforward, we shall only sketch the proofs. We being by stating a simple extension of the (trivial) result for the $K_3$-process mentioned in the Introduction.

\begin{prop}
Let $H = C_k$ for some $k \ge 3$, or $H = K_{2,3}$. Then,
$$p_c(n,H) \,=\, \frac{\log n}{n} \,+\, \Theta\bigg( \frac{1}{n} \bigg).$$
\end{prop}

\begin{proof}[Sketch of proof]
We shall show that, with high probability, the graph $G_{n,p}$ percolates in the $H$-bootstrap process if and only if it is connected. The bounds on $p_c(n,H)$ then follow by standard results, see~\cite{RG}.

Indeed, first let $H = C_k$ and consider a path of length at least $k$ attached to a triangle; we claim that this graph spans a clique (on its vertex set). To see this, identify the vertices with $[\ell]$ so that the edges are $\{ i(i+1) : i \in [\ell-1]\} \cup \{13\}$, and say that $ij$ is a $t$-edge if $|i - j| = t$. The edges are infected in the following order: $(k-1)$-edges, $k$-edges, $2$-edges, $3$-edges, $4$-edges, and so on. Finally, observe that if $G_{n,p}$ is connected then, with high probability, every vertex has a path of length at least $k$ leading to a triangle.

For $H = K_{2,3}$ the proof is similar. Let $x,y \in V(G_{n,p})$, and suppose that there exist vertex disjoint paths from $x$ and $y$ to adjacent vertices of a copy of $C_4$. Then it is easy to see that the percolation process works its way along these paths and eventually infects the edge $xy$. This gives a large complete bipartite graph, and if there is an edge in each part then the closure is a complete graph. Since $G_{n,p}$ is connected, every vertex is eventually swallowed by this clique.
\end{proof}

The case $H = K_{2,3}$ is the first we have seen for which $p_c(n,H) \ne n^{-1/\lambda(H) + o(1)}$. We shall now determine a large family of such graphs. Define
$$\lambda^*(H) \, := \, \min_{e \in E(H)} \max_{F \subset H - e} \left\{ \frac{e(F)}{v(F)} \right\}.$$
This parameter gives us a general lower bound on $p_c(n,H)$.

\begin{prop}\label{genlower}
For every graph $H$, there exists a constant $c(H)$ such that
$$p_c(n,H) \, \ge \, c(H) n^{-1/\lambda^*(H)}$$
for every $n \in \N$.
\end{prop}

\begin{proof}[Sketch of proof]
We shall show that if $p \le c(H) n^{-1/\lambda^*(H)}$ then, with probability at least $1/2$, no new edges are infected in the $H$-bootstrap process. To do so, for each $e \in E(H)$ choose a subgraph $F = F(e) \subset H - e$ which maximizes $e(F) / v(F)$, and note that $e(F)/ v(F) \ge \lambda^*(H)$. Thus, the expected number of copies of $F$ in $G_{n,p}$ is at most
$$n^{v(F)} p^{e(F)} \, \le \, c(H) n^{v(F) - e(H) / \lambda^*(H)} \,\le\, c(H).$$
Summing over edges of $H$, we obtain
$$\Pr\Big( F(e) \subset G_{n,p} \text{ for some }e \in E(H) \Big) \, \le \,  e(H) c(H) \, < \,  \frac{1}{2},$$
if $c(H)$ is sufficiently small. But if $F(e) \not\subset G_{n,p}$ for every $e \in E(H)$ then $H - e \not\subset G_{n,p}$ for every $e \in E(H)$, and hence no new edges are infected, as claimed.
\end{proof}

We next show that Proposition~\ref{genlower} is sharp for a large class of graphs $H$.

\begin{prop}
If $H$ has a leaf, then
$$p_c(n,H) \, = \,  \Theta\left( n^{-1/\lambda^*(H)} \right).$$
\end{prop}

\begin{proof}[Sketch of proof]
The lower bound follows from Proposition~\ref{genlower}. For the upper bound, let $p \gg n^{-1/\lambda^*(H)}$ and recall (see~\cite{RG}) that, with high probability, $H - e \subset G_{n,p}$ for some $e \in E(H)$. (To see this, let $e$ and $F \subset H - e$ be such that $e(F)/v(F) = \max_{F' \subset H - e} e(F')/v(F') = \lambda^*(H)$, find a copy of $F$ in $G_{n,p}$ by the second moment method, and then find $H - e$ by sprinkling.) Let $v_1$ be the neighbour of a leaf in $H$, and observe that we can infect every edge which is incident with $v_1$ (and is not in our copy of $H - e$).

Now, take a second, independent copy of $G_{n,p}$, and apply the same argument inside the neighbourhood of $v_1$. We find a vertex $v_2$ such that we can add (almost) all edges incident with $v_2$. Repeating this process $v(H)$ times, we find (with high probability) a clique on $v(H)$ vertices in $\< G_{n,p^*} \>_H$, where $p^* = v(H) p$.

Finally, observe that $\< K_{v(H)} \>_H = K_n$, since we may add the remaining vertices to the clique one by one. Thus $p = O\big(n^{-1/\lambda^*(H)} \big)$, as claimed.
\end{proof}

A slightly less trivial case, which lies somewhere between a clique and a tree, also matches the general lower bound in Proposition~\ref{genlower}. Say that $H$ is an \emph{$r$-clique-tree} if (for some $2 \le \ell \in \N$) it is composed of $\ell$ vertex-disjoint copies of $K_r$, plus $\ell - 1$ extra edges, and is connected.

\begin{prop}
Let $H$ be an $r$-clique tree. Then
$$c(H) n^{-1/\lambda^*(H)} \, \le \, p_c(n,H) \, \le \, n^{-1/\lambda^*(H)} \log n$$
for some $c(H) > 0$.
\end{prop}

\begin{proof}[Sketch of proof]
The lower bound again follows by Proposition~\ref{genlower}. For the upper bound, we begin by observing that
$$\lambda^*(H) \, = \, \frac{e(H) - 1}{v(H)} \, = \, \frac{{r \choose 2}}{r} + \frac{\ell - 2}{\ell r},$$
where $v(H) = \ell r$. To see this, simply observe that every tree $T$ has a vertex whose removal leaves no component of side larger than $v(T) / 2$, and remove an edge from the corresponding clique; $\lambda^*(H)$ is certainly at least this large since we may always take $F = H - e$.

Assume first that $\ell \ge 3$, and let $p \gg n^{-1/\lambda^*(H)}$ (we shall prove a stronger result in this case). Note that, as in the previous proof, $H - e \subset G_{n,p}$ for some $e \in E(H)$ with high probability; in fact, there exist at least $v(H)$ copies of $H - e$. Moreover, setting $\eps = \frac{\ell - 2}{\ell r}$, there exist at least $n^\eps$ copies of $K_r$ in $G_{n,p}$. Let $X$ denote the union of those copies of $K_r$ which do not intersect a copy of $H - e$.

From each copy of $H - e$, pick a clique $R$ which is the neighbour of a leaf (in the tree-structure of $H$), and observe that we may infect every edge between $R$ and $X$. We thus obtain a complete bipartite graph, with parts of size $v(H)$ and $n^\eps$. Moreover, each part consists of $r$-cliques, and thus these edges span a clique on the same vertex set.

Finally, sprinkling edges with density $p$, we see that every vertex in a copy of $K_r$ minus an edge, and with a neighbour in $X$, is added to the clique. With high probability there are $n^{2\eps}$ such vertices. Repeating this process $1/\eps$ times, we infect the entire edge set, as required.

For the case $\ell = 2$ we prove the weaker bound in the statement. Let $p$ be as above, and take $\log n$ copies of $G_{n,p}$. By the same proof as above, in the first we span a clique of order $C$, for some large constant $C$; in the second a clique of order $C^2$; in the third $C^3$, and so on. In the first step this is just the union of copies of $K_r$; in later steps it is the union of copies of $K_r$ minus an edge which have a neighbour in the clique formed in the previous step. The proposition now follows.
\end{proof}

We give one final cautionary example, whose purpose is just to point out that  $\lambda(H)$ and $\lambda^*(H)$ are not the only possible values of
$$- \ds\lim_{n \to \infty} \frac{\log n}{\log p_c(n,H)}.$$
Let $DD_r$ denote the `double-dumbbell', the graph consisting of two disjoint copies of $K_r$, plus two extra (disjoint) edges between the two cliques. Note that $\lambda(DD_r) = r/2$ and $\lambda^*(DD_r) = (2{r \choose 2} + 1)/2r$, and therefore
$$\lambda^*(DD_r) \,<\, \frac{{r \choose 2} + 1}{r} \,<\, \lambda(DD_r).$$

\begin{prop}
For every $r \ge 4$,
$$- \ds\lim_{n \to \infty} \frac{\log n}{\log p_c(n,DD_r)} \,=\, \frac{{r \choose 2} + 1}{r}.$$
\end{prop}

\begin{proof}[Sketch of proof]
The key observation is that if $H = DD_r$ and $e \in E(DD_r)$, then $\< H - e \>_H = K_{|H|}$, i.e., a copy of $DD_r$ spans a clique on its vertex set. Moreover, two $(\ge 2r)$-cliques which overlap in two (or more) points span a clique on their union. We shall use these observations, plus the usual `critical droplet' argument from bootstrap percolation on $[n]^d$.

We begin with the (easier) upper bound. Let $n^r p^{{r \choose 2} + 1} \gg \log n$, and consider $m = \log n$ copies $G_1,\ldots,G_m$ of $G_{n,p}$. We claim that their union percolates with high probability. To see this, first observe that $G_{n,p}$ contains an $r$-clique $R_1$ with high probability. Next, note that the expected number of copies of $K_r$ plus a pendant edge, with its endpoint in $R_1$, is at least $|R_1| {{n - |R_1|} \choose r} p^{{r \choose 2} + 1} \gg \log n$. Using Chebyshev's inequality, it follows that there exist at least $\log n$ such copies with high probability, and the closure of these is a clique $R_2$ on at least $\log n$ vertices. Now, simply repeat this procedure for each graph $G_3,\ldots,G_m$. A straightforward calculation shows that, with high probability, at each step the clique $R_j$ (at least) doubles in size, until it reaches size $1/p$. But now a positive fraction of the vertices have $r$ neighbours in $R_{m-2}$, so $|R_{m-1}| \ge \eps n$, and thus $|R_m| = n$ with high probability, as required.

To prove the lower bound, we define a process analogous to the Clique Process in Section~\ref{k4sec}. To be precise, we can break up the process into steps of the following two types: $(a)$ if two $(\ge 2r)$-cliques share two vertices then merge them, and $(b)$ if an edge is infected then consider the copy of $H$ it completes, and merge the $(\ge 2r)$-cliques which provided the edges of $H - e$. To see that this works, recall that  $\< DD_r - e \>_{DD_r} = K_{2r}$.

Using this process, we can easily prove a result analogous to Lemma~\ref{ALlemma}, except with $3$ replaced by $e(H)$. Indeed, at each step the size of the largest clique increases by at most a factor of $e(H)$. Moreover, by considering the penultimate step of the process, as in Lemma~\ref{2l-3}, and using induction, we can prove the following extremal result: If $\< G \>_{DD_r} = K_n$ and $n \ge r$, then
$$e(G) \, \ge \, \bigg( \frac{{r\choose 2} + 1}{r} \bigg)n - 1.$$
The result now follows by a straightforward (and standard) calculation (using Markov), as in the proof of Proposition~\ref{k4:lower}.
\end{proof}

We now turn to some open problems. The ultimate aim of this line of research is to understand the $H$-bootstrap process for every graph $H$; a solution to the following problem would represent a major step in this direction.

\begin{prob}\label{prob:allH}
Determine $\ds\lim_{n \to \infty} \frac{\log p_c(n,H)}{\log n}$ for every graph $H$.
\end{prob}

The next problem is probably less difficult, but would still be very interesting. Recall that by a result of Friedgut~\cite{Fri}, together with Theorem~\ref{k=4}, the event $\< G_{n,p} \>_H = K_n$ has a sharp threshold when $H = K_4$, and a coarse threshold when $H = K_r + e$. 

\begin{prob}
Characterize the graphs $H$ for which the event $\< G_{n,p} \>_H = K_n$ has a sharp threshold.
\end{prob}

Returning to cliques, we would also like to have sharper versions of Theorems~\ref{thm:Kr} and~\ref{k=4}.

\begin{prob}
Determine $p_c(n,K_r)$ up to a constant factor.
\end{prob}

\begin{prob}
Find $1/4 \le \alpha \le 24$, if it exists, such that
$$p_c(n,K_4) \,=\, \big( 1 + o(1) \big) \ds\frac{\alpha}{\sqrt{n \log n}}.$$
\end{prob}

Note that the sharpness of the threshold for $H = K_4$ does not imply the existence of such a constant $\alpha$; it would thus be interesting to show that such a constant exists, even without calculating it. 

Since Problem~\ref{prob:allH} is likely to be hard, we mention two natural families of graphs for which we do not have good bounds on the critical probability $p_c(n,H)$: the complete bipartite graphs, and the random graph.

\begin{prob}
Determine $p_c(n,K_{s,t})$, at least up to a poly-logarithmic factor, for all $s,t \in \N$.
\end{prob}

\begin{prob}
Give bounds on $p_c(n,G_{k,1/2})$ which hold with high probability as $k \to \infty$.
\end{prob}

Finally, we mention a substantial generalization of the problem we have considered in this paper. Given graphs $G$ and $H$, define $H$-bootstrap percolation on $G$ by only allowing edges of $G$ to be infected, and say that a graph $F$ percolates if, starting with $F$, eventually all edges of $G$ are infected. (Or, in other words, replace $K_n$ by $G$.) It seems likely that there are many beautiful theorems to discover about this very general bootstrap process.



\end{document}